\documentclass[10pt]{article}

\usepackage{amssymb}
\usepackage{amsmath,amsfonts,mathtools}
\usepackage{amsthm}
\usepackage{latexsym}
\usepackage{mathrsfs}
\usepackage{color}
\usepackage{float}
\usepackage{enumitem}
\usepackage{algorithm,algorithmic}
\usepackage{soul}
 \allowdisplaybreaks[3]

\usepackage[a4paper]{geometry}
\newcommand{\beq}{\begin{equation}}
\newcommand{\eeq}{\end{equation}}
\newcommand{\beqa}{\begin{eqnarray}}
\newcommand{\eeqa}{\end{eqnarray}}
\newcommand{\beqas}{\begin{eqnarray*}}
\newcommand{\eeqas}{\end{eqnarray*}}
\newcommand{\ba}{\begin{array}}
\newcommand{\ea}{\end{array}}
\newcommand{\bi}{\begin{itemize}}
\newcommand{\ei}{\end{itemize}}
\newcommand{\nn}{\nonumber}

\newcommand{\mcX}{{\mathcal X}}
\newcommand{\mcY}{{\mathcal Y}}
\newcommand{\mcL}{{\mathcal L}}

\newcommand{\mcT}{{\mathcal T}}

\newcommand{\prox}{\mathrm{prox}}
\newcommand{\dom}{\mathrm{dom}}
\newcommand{\dist}{\mathrm{dist}}
\newcommand{\argmin}{\arg\min}
\newcommand{\argmax}{\arg\max}

\newtheorem{lemma}{Lemma}
\newtheorem{thm}{Theorem}

\newtheorem{defi}{Definition}

\newtheorem{assumption}{Assumption}

\newtheorem{rem}{Remark}

\newcounter{spb}
\def\cA{{\cal A}}
\def\cB{{\mathbb B}}

\def\cF{{\cal F}}
\def\cI{{\mathscr I}}

\def\cO{{\cal O}}

\def\cA{{\cal A}}

\def\tlambda{{\tilde \lambda}}

\def\wT{{\widehat T}}

\def\tx{{\tilde x}}
\def\ty{{\tilde y}}

\def\bR{{\mathbb{R}}}
\def\halpha{{\hat\alpha}}
\def\hdelta{{\hat\delta}}

\def\xe{{x_\epsilon}}
\def\ye{{y_\epsilon}}

\def\bbK{\mathbb{K}}
\def\bbT{\mathbb{T}}
\def\tl{\tilde\lambda}

\def\bh{{h}}
\def\bH{{H}}
\def\hh{{\bH^*_{\epsilon}}}

\def\h{\bar h}
\def\H{\bar H}
\def\cP{{\cal P}}
\def\cG{{\cal G}}

\def\AL{{\mcL}}

\def\np{{p}}
\def\nq{{q}}

\def\bff{{\rm \bf nf}}
\def\bfx{{\rm \bf x}}
\def\bfy{{\rm \bf y}}

\def\tm{{\tilde m}}
\def\tn{{\tilde n}}

\definecolor{ngreen}{RGB}{38,217,169}

\title{A first-order method for nonconvex-strongly-concave constrained minimax optimization\thanks{This work was partially supported by the Office of Naval Research under Award N00014-24-1-2702,  the Air Force Office of Scientific Research under Award FA9550-24-1-0343, and the National Science Foundation under Award IIS-2211491. It was primarily conducted during Sanyou Mei’s Ph.D. studies at the University of Minnesota.}}

\author{
Zhaosong Lu
\thanks{
Department of Industrial and Systems Engineering, University of Minnesota, USA (email: {\tt zhaosong@umn.edu}).}
\and
Sanyou Mei
\thanks{
Department of Industrial Engineering and Decision Analytics, the Hong Kong University of Science and Technology, Hong Kong, China (email: {\tt symei@ust.hk}).}
}

\date{May 12, 2024 (Revised: October 23, 2025)}

\begin{document}
\maketitle

\begin{abstract}
In this paper we study a nonconvex-strongly-concave constrained minimax problem. Specifically, we propose a first-order augmented Lagrangian method for solving it, whose subproblems are nonconvex-strongly-concave unconstrained minimax problems and suitably solved by a first-order method developed in this paper that leverages the strong concavity structure. Under suitable assumptions, the proposed method achieves an \emph{operation complexity} of $\cO(\varepsilon^{-3.5}\log\varepsilon^{-1})$, measured in terms of its fundamental operations, for finding an $\varepsilon$-KKT solution of the constrained minimax problem, which improves the previous best-known operation complexity by a factor of $\varepsilon^{-0.5}$.
\end{abstract}


\medskip


\section{Introduction}

In this paper, we consider a nonconvex-strongly-concave constrained minimax problem
\beq\label{prob}
F^*=\min_{c(x)\leq 0}\max_{d(x,y)\leq 0}\{F(x,y):=f(x,y)+p(x)-q(y)\}.
\eeq
For notational convenience, throughout this paper we let $\mcX:=\dom\,p$ and $\mcY:=\dom\,q$, where $\dom\,p$ and $\dom\,q$ denote the domain of $p$ and $q$, respectively. 
Assume that problem \eqref{prob} has at least one optimal solution and the following additional assumptions hold.
\begin{assumption}\label{a1}
\begin{enumerate}[label=(\roman*)]
\item $f$ is $L_{\nabla f}$-smooth on $\mcX\times\mcY$, and $f(x,\cdot)$ is $\sigma$-strongly-concave for some constant $\sigma>0$ for any given $x\in\mcX$.\footnote{The definition of $L_F$-Lipschitz continuity, $L_{\nabla f}$-smoothness and $\sigma$-strongly-concavity is given in Subsection \ref{notation}.}
\item $p:\bR^n\to\bR\cup\{+\infty\}$ and $q:\bR^m\to\bR\cup\{+\infty\}$ are proper closed convex functions, and the proximal operators of $p$ and $q$ can be exactly evaluated.
\item $c:\bR^n\to\bR^{\tn}$ is $L_{\nabla c}$-smooth and $L_c$-Lipschitz continuous on $\mcX$, $d:\bR^n\times\bR^m\to\bR^{\tm}$ is $L_{\nabla d}$-smooth and $L_d$-Lipschitz continuous on $\mcX\times\mcY$, and $d_i(x,\cdot)$ is convex for each $x\in\mcX$.
\item The sets $\mcX$ and $\mcY$ (namely, $\dom\,p$ and $\dom\,q$) are compact.
\end{enumerate}
\end{assumption}
Problem \eqref{prob} has found application in various areas, such as perceptual adversarial robustness \cite{laidlaw2020perceptual}, robust adversarial classification \cite{ho2023adversarial}, adversarial attacks in resource allocation \cite{tsaknakis2023minimax}, network interdiction problem \cite{fu2019network,smith2013modern}, and power networks \cite{salmeron2004analysis}.

In recent years, the minimax problem of a simpler form has gained significant attention:
\beq \label{special-minimax}
\min_{x\in X}\max_{y\in Y} f(x; y),
\eeq
where $X$ and $Y$ are closed sets. This problem has found wide applications in various areas,  including adversarial training \cite{Good15,Madry18,Sin18,Wang21}, generative adversarial networks \cite{Gid19,Good14,sanj18}, reinforcement learning \cite{Dai18,Du17,Na19,qiu20,song18}, computational game \cite{Ant21,Rak13,Syr15}, distributed computing \cite{Mat10,Sha08},  prediction and regression \cite{Ces06,Tas06,Xu09,XuNe05}, and distributionally robust optimization \cite{Duc19,Sha15}. Numerous methods have been developed to solve problem \eqref{special-minimax} when $X$ and $Y$ are \emph{simple closed convex sets} (e.g., see \cite{chen2021proximal,guo2023fast,huang2020accelerated,lin20b,lin2020near,lu2020hybrid,
luo2020stochastic,nou19,ostrovskii2021efficient,xian2021faster,xu2020gradient,xu2023unified,
Yang20,zhang2020single}). 
In addition, first-order methods were developed in \cite{kong2021accelerated, zhao2024primal} for solving problem \eqref{prob} with $c(x)\equiv0$ and $d(x,y)\equiv0$. 

There have also been several studies on other special cases of problem \eqref{prob}. Specifically, in \cite{goktas2021convex}, two first-order methods called max-oracle gradient-descent and nested gradient descent/ascent methods were proposed for solving \eqref{prob}. These methods assume that $c(x)\equiv0$ and $p$ and $q$ are the indicator function of simple compact convex sets $X$ and $Y$, respectively. They also require the convexity of $V(x)=\max_{y\in Y} \{f(x,y): d(x,y)\leq0\}$, as well as the ability to compute an optimal Lagrangian multiplier associated with the constraint $d(x,y)\leq 0$ for each $x\in X$. Moreover, in \cite{dai2022rate}, an augmented Lagrangian (AL) method was recently proposed for solving \eqref{prob} with only equality constraints, $p(x)\equiv0$, $q(y)\equiv0$ and $c(x)\equiv0$.  This method assumes that a local min-max point of the AL subproblem can be found at each iteration. Furthermore, \cite{tsaknakis2023minimax}  introduced a multiplier gradient descent method for solving \eqref{prob} with $c(x)\equiv0$, $d(x,y)$ being an affine mapping, and $p$ and $q$ being the indicator function of a simple compact convex set. In addition, \cite{dai2024optimality} developed a proximal gradient multi-step ascent-decent method for problem \eqref{prob} with $c(x)\equiv0$, $d(x,y)$ being an affine mapping, and $f(x,y)=g(x)+x^TAy-h(y)$, assuming that $f(x,y)-q(y)$ is \emph{strongly concave} in $y$. Furthermore, primal dual alternating proximal gradient methods were proposed in \cite{zhang2022primal} for solving \eqref{prob} under the conditions of $c(x)\equiv0$, $d(x,y)$ being an affine mapping, and either $f(x,y)$ being strongly concave in $y$ or [$q(y)\equiv0$ and $f(x,y)$ being a linear function in $y$]\}. While the aforementioned studies \cite{dai2024optimality,goktas2021convex,zhang2022primal} established the iteration complexity of the methods for finding an approximate stationary point of a special minimax problem, the operation complexity, measured by fundamental operations such as gradient evaluations of $f$ and proximal operator evaluations of $p$ and $q$, was not studied in these works.
 
Recently, a first-order augmented Lagrangian (AL) method was proposed in \cite[Algorithm 3]{lu2024first} for solving a nonconvex-concave constrained minimax problem in the form of  \eqref{prob} in which $f(x,\cdot)$ is however merely concave for any given $x\in\mcX$. Under suitable assumptions, this method achieves an operation complexity of $\cO(\varepsilon^{-4}\log\varepsilon^{-1})$, measured by the amount of evaluations of $\nabla f$, $\nabla c$, $\nabla d$ and proximal operators of $p$ and $q$, for finding an $\varepsilon$-KKT solution of the problem. While this method is applicable to problem \eqref{prob}, it does not exploit the strong concavity structure of $f(x,\cdot)$. Consequently, it may not be the most efficient method for solving \eqref{prob}. 

In this paper, we propose a first-order AL method for solving problem \eqref{prob}. Our approach follows a similar framework as \cite[Algorithm 3]{lu2024first}, but we enhance it by leveraging the strong concavity of $f(x,\cdot)$. As a result, our method achieves a substantially improved operation complexity compared to \cite[Algorithm 3]{lu2024first}. Specifically, given an iterate $(x^k,y^k)$ and a Lagrangian multiplier estimate $(\lambda^k_\bfx,\lambda^k_\bfy)$ at the $k$th iteration, the next iterate $(x^{k+1},y^{k+1})$ of our method is obtained by finding an approximate stationary point of the AL subproblem
\beq \label{minimax-AL}
\min_x\max_y \AL(x,y,\lambda^k_\bfx,\lambda^k_\bfy;\rho_k)
\eeq 
for some $\rho_k>0$, 
where $\AL$ is the AL function of \eqref{prob} defined as
\beq\label{AL}
\AL(x,y,\lambda_\bfx,\lambda_\bfy;\rho)=F(x,y)+\frac{1}{2\rho}\left(\|[\lambda_\bfx+\rho c(x)]_+\|^2-\|\lambda_\bfx\|^2\right)-\frac{1}{2\rho}\left(\|[\lambda_\bfy+\rho d(x,y)]_+\|^2-\|\lambda_\bfy\|^2\right),
\eeq
which is a generalization of the AL function introduced in \cite{dai2022rate} for an equality constrained
minimax problem. The Lagrangian multiplier estimate is then updated by $\lambda_\bfx^{k+1}=\Pi_{\cB^+_\Lambda}(\lambda^k_\bfx+\rho_kc(x^{k+1}))$ and $\lambda^{k+1}_\bfy=[\lambda^k_\bfy+\rho_kd(x^{k+1},y^{k+1})]_+$  for some $\Lambda>0$, where $\Pi_{\cB^+_\Lambda}(\cdot)$ and $[\cdot]_+$ are defined in Subsection \ref{notation}. Given that  problem \eqref{minimax-AL} is a nonconvex-strongly-concave unconstrained minimax problem,  we develop an efficient first-order method for finding an approximate stationary point of it by utilizing its strong concavity structure.

\vspace{.1in}

The main contributions of this paper are summarized below.

\begin{itemize}
\item  We propose a first-order method for solving a nonconvex-strongly-concave unconstrained minimax problem. Under suitable assumptions, we show that this method achieves an operation complexity of $\cO(\varepsilon^{-2}\log\varepsilon^{-1})$, measured by its fundamental operations, for finding an $\varepsilon$-primal-dual stationary point of the problem, which improves the previous best-known operation complexity achieved by \cite[Algorithm 1]{lu2024first} by a factor of $\varepsilon^{-0.5}$.
\item  We propose a first-order AL method for solving nonconvex-strongly-concave constrained minimax problem \eqref{prob}. Under suitable assumptions, we show that this method achieves an operation complexity of $\cO(\varepsilon^{-3.5}\log\varepsilon^{-1})$, measured by its fundamental operations, for finding an $\varepsilon$-KKT solution of \eqref{prob}, which improves the previous best-known operation complexity achieved by \cite[Algorithm 3]{lu2024first} by a factor of $\varepsilon^{-0.5}$.
\end{itemize}

The rest of this paper is organized as follows. In Subsection \ref{notation}, we introduce some notation and terminology. In Section \ref{minimax}, we propose a first-order method for solving a nonconvex-concave minimax problem and study its complexity. In Section~\ref{sec:main}, we propose a first-order AL method for solving problem \eqref{prob} and present complexity results for it. Finally, we provide the proof of the main results in Section \ref{sec:proofs}.

\subsection{Notation and terminology}  \label{notation}
The following notation will be used throughout this paper. Let $\bR^n$ denote the Euclidean space of dimension $n$ and $\bR^n_+$ denote the nonnegative orthant in $\bR^n$. The standard inner product, $l_1$-norm and Euclidean norm are denoted by $\langle\cdot,\cdot\rangle$, $\|\cdot\|_1$ and $\|\cdot\|$, respectively. For any $\Lambda>0$, let $\cB^+_\Lambda =\{x \geq 0:\|x\|\leq \Lambda\}$, whose dimension is clear from the context. 
For any $v\in\bR^n$, let $v_+$ denote the nonnegative part of $v$, that is, $(v_+)_i=\max\{v_i,0\}$ for all $i$. 
 Given a point $x$ and a closed set $S$ in $\bR^n$, let 
$\dist(x,S)=\min_{x'\in S} \|x'-x\|$, $\Pi_{S}(x)$ denote the Euclidean projection of $x$ onto $S$, and $\cI_S$ denote the indicator function associated with $S$.

A function or mapping $\phi$ is said to be \emph{$L_{\phi}$-Lipschitz continuous} on a set $S$ if $\|\phi(x)-\phi(x')\| \leq L_{\phi} \|x-x'\|$ for all $x,x'\in S$. In addition, it is said to be \emph{$L_{\nabla\phi}$-smooth} on $S$ if $\|\nabla\phi(x)-\nabla\phi(x')\| \leq L_{\nabla\phi} \|x-x'\|$ for all $x,x'\in S$. 
A function is said to be $\sigma$-strongly-convex if it is strongly convex with modulus $\sigma>0$. For a closed convex function $p:\bR^n\to \bR\cup\{+\infty\}$, the \emph{proximal operator} associated with $p$ is denoted by  
$\prox_p$,  that is,
\[
\prox_p(x) = \argmin_{x'\in\bR^n} \left\{ \frac{1}{2}\|x' - x\|^2 + p(x') \right\} \quad \forall x \in \bR^n.
\]
Given that evaluation of $\prox_{\gamma p}(x)$ is often as cheap as $\prox_p(x)$, we count the evaluation of $\prox_{\gamma p}(x)$ as one evaluation of proximal operator of $p$ for any $\gamma>0$ and $x\in\bR^n$. 

For a lower semicontinuous function $\phi:\bR^n\to \bR\cup\{+\infty\}$, its \emph{domain} is the set $\dom\, \phi := \{x| \phi(x)<+\infty\}$. The \emph{upper subderivative} of $\phi$ at $x\in \dom\, \phi$ in a direction $d\in\bR^n$ is defined by
\[
\phi'(x;d) = \limsup\limits_{x' \stackrel{\phi}{\to} x,\, t \downarrow 0} \inf_{d' \to d} \frac{\phi(x'+td')-\phi(x')}{t},
\] 
where $t\downarrow 0$ means both $t > 0$ and $t\to 0$, and $x' \stackrel{\phi}{\to} x$ means both $x' \to x$ and $\phi(x')\to \phi(x)$. The \emph{subdifferential} of $\phi$ at $x\in \dom\, \phi$ is the set 
\[
\partial \phi(x) = \{s\in\bR^n\big| s^T d \leq \phi'(x; d) \ \ \forall d\in\bR^n\}.
\]
We use $\partial_{x_i} \phi(x)$ to denote the subdifferential with respect to $x_i$.  
In addition, for an upper semicontinuous function $\phi$, its subdifferential is defined as $\partial \phi=-\partial (-\phi)$. If $\phi$ is locally Lipschitz continuous, the above definition of subdifferential coincides with the Clarke subdifferential. Besides, if $\phi$ is convex, it coincides with the ordinary subdifferential for convex functions. Also, if $\phi$ is continuously differentiable at $x$ , we simply have $\partial \phi(x) = \{\nabla \phi(x)\}$, where $\nabla \phi(x)$ is the gradient of $\phi$ at $x$. In addition, it is not hard to verify that $\partial (\phi_1+\phi_2)(x)=\nabla \phi_1(x)+\partial \phi_2(x)$ if $\phi_1$ is continuously differentiable at $x$ and $\phi_2$ is lower or upper semicontinuous at $x$. See \cite{clarke1990optimization,ward1987nonsmooth} for more details.

Finally, we introduce an (approximate) primal-dual stationary point (e.g., see \cite{dai2024optimality,dai2020optimality,kong2021accelerated}) for a general minimax problem
\begin{equation}\label{eg}
\min_{x}\max_{y}\Psi(x,y),
\end{equation}
where $\Psi(\cdot,y): \bR^n \to \bR \cup\{+\infty\}$ is a lower semicontinuous function, and $\Psi(x,\cdot): \bR^m \to \bR \cup\{-\infty\}$ is an upper semicontinuous function.

\begin{defi}  \label{def2}
A point $(x,y)$ is said to be a primal-dual stationary point of the minimax problem \eqref{eg} if 
\[
0 \in \partial_x\Psi(x,y), \quad 0\in\partial_y\Psi(x,y).
\]
In addition, for any $\epsilon>0$, a point $(\xe,\ye)$ is said to be an $\epsilon$-primal-dual stationary point of the minimax problem \eqref{eg} if
\begin{equation*}
\dist\left(0,\partial_x\Psi(\xe,\ye)\right)\leq\epsilon,\quad\dist\left(0,\partial_y\Psi(\xe,\ye)\right)\leq\epsilon.
\end{equation*}
\end{defi}

One can see that $(\xe,\ye)$ is an $\epsilon$-primal-dual stationary point of  \eqref{eg} if and only if 
$\xe$ and $\ye$ are an $\epsilon$-stationary point of $\min_{x}\Psi(x,\ye)$ and $\max_{y}\Psi(\xe,y)$, respectively. 

\section{A first-order method for nonconvex-strongly-concave unconstrained minimax optimization} 
\label{minimax}

In this section, we propose a first-order method for finding an $\epsilon$-primal-dual stationary point of a  nonconvex-strongly-concave unconstrained minimax problem, which will be used as a subproblem solver for the first-order AL method proposed in Section \ref{sec:main}. In particular, we consider a nonconvex-strongly-concave minimax problem
\begin{equation}\label{mmax-prob}
\bH^* = \min_x\max_y\left\{\bH(x,y)\coloneqq \bh(x,y)+\np(x)-\nq(y)\right\}.
\end{equation}
Assume that problem \eqref{mmax-prob} has at least one optimal solution and $p, q$ satisfy Assumption \ref{a1}.  In addition, $\bh$ satisfies the following assumption.

\begin{assumption}\label{mmax-a}
The function $\bh$ is $L_{\nabla\bh}$-smooth on $\mcX\times\mcY$, and moreover, $h(x,\cdot)$ is $\sigma_y$-strongly-concave for some constant $\sigma_y>0$ for all $x\in\mcX$, where $\mcX:=\dom\,p$ and $\mcY:=\dom\,q$.
\end{assumption}

Several first-order methods  have been developed for special classes of \eqref{mmax-prob} with $p, q$ being the indicator function of convex compact sets or entire spaces, and they enjoy  an operation complexity of $\cO(\epsilon^{-2}\log \epsilon^{-1})$, measured by the amount of evaluations of $\nabla \bh$ and proximal operators of $p$ and $q$, for finding an $\epsilon$-primal-dual stationary point of  \eqref{mmax-prob} with such $p$ and $q$ (e.g., see \cite{lin2020near,Yang20}).  They are however not applicable to \eqref{mmax-prob} in general. 

We now propose a first-order method for problem \eqref{mmax-prob} by solving a sequence of subproblems
\beq\label{ppa-subprob}
 \min_x\max_y\left\{\bH_k(x,y):=\bh_k(x,y)+\np(x)-\nq(y)\right\}, 
 \eeq
which result from applying an inexact proximal point method \cite{kaplan1998proximal} to the minimization problem \\ 
$\min_x\{\max_y \bh(x,y)+\np(x)-\nq(y)\}$,  where 
 \begin{equation}\label{mmax-sub}
\bh_k(x,y)=\bh(x,y)+L_{\nabla \bh}\|x-x^k\|^2,
\end{equation}
and $x^k$ is an approximate $x$-solution of \eqref{ppa-subprob} with $k$ replaced by $k-1$. 
By Assumption \ref{mmax-a}, one can observe that (i) $\bh_k$ is $L_{\nabla \bh}$-strongly convex in $x$ and $\sigma_y$-strongly concave in $y$ on $\dom\,\np\times\dom\,\nq$; (ii) $\bh_k$ is $3L_{\nabla \bh}$-smooth on $\dom\,\np\times\dom\,\nq$. Consequently, problem \eqref{ppa-subprob} is a special case of \eqref{ea-prob} and can be suitably solved by Algorithm~\ref{mmax-alg1} (see Appendix \ref{strong-cvx-ccv}). The resulting first-order method for \eqref{mmax-prob} is presented in Algorithm \ref{mmax-alg2}.

\begin{algorithm}[H]
\caption{A first-order method for problem~\eqref{mmax-prob}}
\label{mmax-alg2}
\begin{algorithmic}[1]
\REQUIRE $\epsilon>0$, $\hat\epsilon_0\in(0,\epsilon/2]$, 
$(\hat x^0,\hat y^0)\in\dom\,\np\times\dom\,\nq$, $(x^0,y^0)=(\hat x^0,\hat y^0)$, 
and $\hat\epsilon_k=\hat\epsilon_0/(k+1)$.
\FOR{$k=0,1,2,\ldots$}
\STATE Call Algorithm~\ref{mmax-alg1} (see Appendix \ref{strong-cvx-ccv}) with $\h\leftarrow \bh_k$, $\bar\epsilon \leftarrow \hat\epsilon_k$, $\sigma_x\leftarrow L_{\nabla \bh}$, $\sigma_y\leftarrow \sigma_y$, $L_{\nabla \h}\leftarrow 3L_{\nabla \bh}$, $\bar z^0=z^0_f\leftarrow-\sigma_x x^k$, $\bar y^0=y^0_f\leftarrow y^k$, and denote its output by $(x^{k+1},y^{k+1})$, where $\bh_k$ is given in \eqref{mmax-sub}.
\STATE Terminate the algorithm and output $(\xe,\ye)=(x^{k+1},y^{k+1})$ if
\begin{equation}\label{mmax-term}
\|x^{k+1}-x^k\|\leq\epsilon/(4L_{\nabla \bh}).
\end{equation}
\ENDFOR
\end{algorithmic}
\end{algorithm}
\begin{rem}\label{ppa-rem}
It is seen from step 2 of Algorithm~\ref{mmax-alg2} that $(x^{k+1},y^{k+1})$ results from applying Algorithm~\ref{mmax-alg1} to the subproblem \eqref{ppa-subprob}. As will be shown in Lemma~\ref{innercplx-alg6}, $(x^{k+1},y^{k+1})$ is an $\hat\epsilon_k$-primal-dual stationary point of \eqref{ppa-subprob}.   
\end{rem}

We next study complexity of Algorithm \ref{mmax-alg2} for finding an $\epsilon$-primal-dual stationary point of problem~\eqref{mmax-prob}. Before proceeding, we define
\begin{align}
&D_\bfx\coloneqq \max\{\|u-v\|\big|u,v\in\mcX\},\quad D_\bfy\coloneqq\max\{\|u-v\|\big|u,v\in\mcY\}, \label{mmax-D}\\
&\bH_{\rm low}:=\min\left\{\bH(x,y)|(x,y)\in\dom\,\np\times\dom\,\nq\right\}.\label{mmax-bnd}
\end{align}
By Assumption~\ref{a1}, one can observe that $\bH_{\rm low}$ is finite.

The following theorem presents \emph{iteration and operation complexity} of Algorithm~\ref{mmax-alg2} for finding an $\epsilon$-primal-dual stationary point of problem \eqref{mmax-prob}, whose proof is deferred to Subsection \ref{sec:proof2-2}.

\begin{thm}[{\bf Complexity of Algorithm \ref{mmax-alg2}}]\label{mmax-thm}
Suppose that Assumption~\ref{mmax-a} holds. Let $\bH^*$, $H$ $D_\bfx$, $D_\bfy$, and $\bH_{\rm low}$ be defined in \eqref{mmax-prob}, 
\eqref{mmax-D} and \eqref{mmax-bnd}, $L_{\nabla \bh}$ be given in Assumption \ref{mmax-a}, $\epsilon$, $\hat\epsilon_0$ and $\hat x^0$ be given in Algorithm~\ref{mmax-alg2}, and 
\begin{align}
\halpha=&\ \min\left\{1,\sqrt{8\sigma_y/L_{\nabla \bh}}\right\},\label{mmax-balpha}\\
\hdelta=&\ (2+\halpha^{-1})L_{\nabla \bh} D_\bfx^2+\max\left\{2\sigma_y,\halpha L_{\nabla \bh}/4\right\}D_\bfy^2,\label{mmax-tP}\\
\wT=&\ \left\lceil16(\max_y\bH(\hat x^0,y)-\bH^*)L_{\nabla \bh}\epsilon^{-2}+32\hat\epsilon_0^2(1+\sigma_y^{-2}L_{\nabla \bh}^2)\epsilon^{-2}-1\right\rceil_+,\label{mmax-K}\\
\widehat N=&\ 3397\max\left\{2,\sqrt{L_{\nabla \bh}/(2\sigma_y)}\right\}\notag\\
&\ \times\Bigg[(\wT+1)\Bigg(\log\frac{4\max\left\{\frac{1}{2L_{\nabla \bh}},\min\left\{\frac{1}{2\sigma_y},\frac{4}{\halpha L_{\nabla \bh}}\right\}\right\}\left(\hdelta+2\halpha^{-1}(\bH^*-\bH_{\rm low}+L_{\nabla \bh} D_\bfx^2)\right)}{(9L_{\nabla \bh}^2/\min\{L_{\nabla \bh},\sigma_y\}+ 3L_{\nabla \bh})^{-2}\hat\epsilon_0^2}\Bigg)_+\notag\\
&\qquad\, +\wT+1+2\wT\log(\wT+1) \Bigg].\label{mmax-N-old}
\end{align}
Then Algorithm~\ref{mmax-alg2} terminates and outputs an $\epsilon$-primal-dual stationary point $(\xe,\ye)$ of \eqref{mmax-prob} in at most $\wT+1$ outer iterations that satisfies 
\begin{equation}\label{upperbnd-old}
\max_y\bH(\xe,y)\leq \max_y\bH(\hat x^0,y)+2\hat\epsilon_0^2\left(L_{\nabla \bh}^{-1}+\sigma_y^{-2}L_{\nabla \bh}\right).
\end{equation}
Moreover, the total number of evaluations of $\nabla \bh$ and proximal operators of $p$ and $q$ performed in Algorithm~\ref{mmax-alg2} is no more than $\widehat N$, respectively.
\end{thm}

\begin{rem}
One can observe from Theorem~\ref{mmax-thm} that $\halpha = \cO(\kappa^{-1/2})$, $\hdelta = \cO(\kappa^{1/2})$, $\wT = \cO(\epsilon^{-2})$, and $\widehat N = \cO(\kappa^{1/2}\epsilon^{-2} \log \hat\epsilon_0^{-1})$, where $\kappa = L_{\nabla h} / \sigma_y$ is the condition number of the maximization part. Consequently, by setting $\hat\epsilon_0 = \epsilon/2$, Algorithm~\ref{mmax-alg2} achieves an operation complexity of $\cO(\kappa^{1/2} \epsilon^{-2} \log \epsilon^{-1})$, measured by the number of evaluations of $\nabla \bh$ and the proximal operators of $p$ and $q$, for computing an $\epsilon$-primal-dual stationary point of the nonconvex–strongly-concave minimax problem~\eqref{mmax-prob}. This improves the best-known complexity bound previously obtained by \cite[Algorithm 1]{lu2024first} by a factor of $\epsilon^{-1/2}$. In addition, an alternating gradient projection (AGP) method was recently proposed in \cite{xu2023unified} for a subclass of unconstrained minimax problems of the form~\eqref{mmax-prob}, specifically those where $p$ and $q$ are indicator functions of convex compact sets. A complexity bound is established for AGP in terms of the norm of a gradient mapping, which has slightly better dependence on $\epsilon$ (up to a logarithmic factor) than our result. However, it has significantly worse dependence on the condition number $\kappa$ due to the lack of an acceleration scheme in AGP.
\end{rem}

\section{A first-order augmented Lagrangian method for nonconvex-strongly-concave constrained minimax optimization}\label{sec:main}

In this section, we propose a first-order augmented Lagrangian (FAL) method in Algorithm \ref{AL-alg} for problem \eqref{prob}, and study its complexity for finding an approximate KKT point of \eqref{prob}. The proposed FAL method follows a similar framework as \cite[Algorithm 3]{lu2024first}. Specifically, at each iteration, the FAL method finds an approximate primal-dual stationary point of an AL subproblem in the form of
\beq\label{AL-sub0}
\min_x\max_y\AL(x,y,\lambda_\bfx,\lambda_\bfy;\rho),
\eeq
where $\AL$ is the AL function associated with problem \eqref{prob} defined in \eqref{AL}, $\lambda_\bfx\in\bR_+^{\tn}$ and $\lambda_\bfy\in\bR_+^{\tm}$ are Lagrangian multiplier estimates, and $\rho>0$ is a penalty parameter, which are updated by a standard scheme. By Assumption \ref{a1}, it is not hard to observe that \eqref{AL-sub0} is a special case of nonconvex-strongly-concave unconstrained minimax problem \eqref{mmax-prob}. Consequently, our FAL method applies Algorithm \ref{mmax-alg2} to find an approximate primal-dual stationary point of \eqref{AL-sub0}.

Before presenting the FAL method for \eqref{prob}, we let
\begin{align}
&\AL_\bfx(x,y,\lambda_\bfx;\rho):=F(x,y)+\frac{1}{2\rho}\left(\|[\lambda_\bfx+\rho c(x)]_+\|^2-\|\lambda_\bfx\|^2\right),
\nn \\
&c_{\rm hi}:=\max\{\|c(x)\|\big|x\in\mcX\},\quad d_{\rm hi}:=\max\{\|d(x,y)\|\big|(x,y)\in\mcX\times\mcY\}, \label{cdhi} 
\end{align}
where $\AL_\bfx(\cdot,y,\lambda_\bfx;\rho)$ can be viewed as the AL function for the minimization part of \eqref{prob}, namely, the problem $\min_x \{F(x,y)| c(x)\leq 0\}$ for any $y\in\mcY$. Besides, we make one additional assumption below regarding the availability of a nearly feasible point for the minimization part of \eqref{prob}. Given the possible nonconvexity of $c_i$'s, it will be used to specify an initial point for solving the AL subproblems (see step 2 of Algorithm \ref{AL-alg}) so that the resulting FAL method outputs an approximate KKT point of  \eqref{prob} nearly satisfying the constraint $c(x) \leq 0$.

\begin{assumption}\label{knownfeas}		
For any given $\varepsilon\in (0,1)$, a $\sqrt{\varepsilon}$-nearly feasible point $x_\bff$ of problem~\eqref{prob}, namely $x_\bff\in\mcX$ satisfying $\|[c(x_\bff)]_+\|\le\sqrt{\varepsilon}$, can be found.	
\end{assumption}

\begin{rem}
A very similar assumption as Assumption~\ref{knownfeas} was considered in \cite{CGLY17,GY19,lu2024first,LZ12,XW19}. In addition,  when the error bound condition $\|[c(x)]_+\|=\cO(\dist(0,\partial (\|[c(x)]_+\|^2+\cI_{\mcX}(x))))^\nu)$ holds on a level set of $\|[c(x)]_+\|$ for some $\nu>0$, Assumption~\ref{knownfeas} holds for problem~\eqref{prob} (e.g., see \cite{lu2022single,S19iAL}). In this case, one can find the above $x_\bff$ by applying a projected gradient method to the problem $\min_{x\in\mcX}\|[c(x)]_+\|^2$.
\end{rem}

 We are now ready to present the aforementioned FAL method for solving problem \eqref{prob}.
 
\begin{algorithm}[H]
\caption{A first-order augmented Lagrangian method for problem \eqref{prob}}\label{AL-alg}
\begin{algorithmic}[1]
\REQUIRE $\varepsilon, \tau\in(0,1)$, $\epsilon_k=\tau^k$, $\rho_k=\epsilon_k^{-1}$, $\Lambda>0$, $\lambda_\bfx^0\in\cB^+_\Lambda$, $\lambda_\bfy^0\in\bR_+^{\tm}$, $(x^0,y^0)\in\dom\,p\times\dom\,q$, and $x_\bff\in\dom\,p$ with $\|[c(x_\bff)]_+\|\leq\sqrt{\varepsilon}$. 
\FOR{$k=0,1,\dots$}
\STATE Set
\begin{align*} \label{xinit}
x^k_{\rm init}=\left\{\begin{array}{ll}
x^k,&\quad \mbox{if }\AL_\bfx(x^k,y^k,\lambda^k_\bfx;\rho_k)\leq\AL_\bfx(x_\bff,y^k,\lambda^k_\bfx;\rho_k),\\
x_\bff,&\quad \mbox{otherwise.}
\end{array}\right.
\end{align*}
\STATE Call Algorithm \ref{mmax-alg2} with $\epsilon\leftarrow\epsilon_k$, $\hat\epsilon_0\leftarrow\epsilon_k/2$, $(x^0,y^0)\leftarrow (x^k_{\rm init},y^k)$, $\sigma_y\leftarrow\sigma$ and $L_{\nabla h}\leftarrow L_k$ to find an $\epsilon_k$-primal-dual stationary point $(x^{k+1},y^{k+1})$ of 
\beq\label{AL-sub}
\min_x\max_y\AL(x,y,\lambda^k_\bfx,\lambda^k_\bfy;\rho_k)
\eeq
 where
\beq\label{Lk}
L_k= L_{\nabla f}+\rho_kL_c^2+\rho_kc_{\rm hi}L_{\nabla c}+\|\lambda^k_\bfx\|L_{\nabla c}+\rho_kL_d^2+\rho_kd_{\rm hi}L_{\nabla d}+\|\lambda^k_\bfy\|L_{\nabla d}.
\eeq
\STATE Set $\lambda_\bfx^{k+1}=\Pi_{\cB^+_\Lambda}(\lambda^k_\bfx+\rho_kc(x^{k+1}))$ and $\lambda^{k+1}_\bfy=[\lambda^k_\bfy+\rho_kd(x^{k+1},y^{k+1})]_+$.
\STATE If $\epsilon_k\leq\varepsilon$, terminate the algorithm and output $(x^{k+1},y^{k+1})$. 
\ENDFOR
\end{algorithmic}
\end{algorithm}

\begin{rem}
\begin{enumerate}[label=(\roman*)]	
\item $\lambda^{k+1}_\bfx$ results from projecting onto a nonnegative Euclidean ball the standard Lagrangian multiplier estimate $\tilde\lambda_\bfx^{k+1}$ obtained by the classical scheme $\tilde\lambda_\bfx^{k+1}=[\lambda^k_\bfx+\rho_kc(x^{k+1})]_+$. It is called a safeguarded Lagrangian multiplier in the relevant literature \cite{BM14,BM20,KS17example}, which has been shown to enjoy many practical and theoretical advantages (see \cite{BM14} for discussions).
\item In view of Theorem \ref{mmax-thm}, one can see that an $\epsilon_k$-primal-dual stationary point of \eqref{AL-sub} can be successfully found in step 3 of Algorithm~\ref{AL-alg} by applying Algorithm \ref{mmax-alg2} to problem \eqref{AL-sub}. Consequently, Algorithm~\ref{AL-alg} is well-defined.
\end{enumerate}
\end{rem}

In the remainder of this section, we study iteration and operation complexity for Algorithm \ref{AL-alg}. Recall that $\mcX=\dom\,p$ and $\mcY=\dom\,q$. To proceed, we make one additional assumption that a generalized Mangasarian-Fromowitz constraint qualification (GMFCQ) holds for the minimization part of \eqref{prob},
a uniform Slater's condition holds for the maximization part of \eqref{prob}, and $F(\cdot,y)$ is Lipschitz continuous on $\mcX$ for any $y\in\mcY$. Specifically, GMFCQ and the Lipschitz continuity of $F(\cdot,y)$ will be used to bound the amount of violation on feasibility and complementary slackness by $(x^{k+1}, \tl^{k+1}_\bfx)$ for the minimization part of \eqref{prob} with $\tl^{k+1}_\bfx=[\lambda^k_\bfx+\rho_kc(x^{k+1})]_+$ (see Lemma \ref{l-xcnstr2}). Likewise, the uniform Slater's condition will be used to bound the amount of violation on feasibility and complementary slackness by $(x^{k+1}, y^{k+1},\lambda^{k+1}_\bfy)$ for the maximization part of \eqref{prob} (see Lemmas \ref{l-ycnstr} and \ref{l-subdcnstr}). 

\begin{assumption}\label{mfcq}
\bi
\item[(i)] There exist some constants $\delta_c$, $\theta>0$ such that for each $x\in\cF(\theta)$ there exists some $v_x \in\mcT_{\mcX}(x)$  satisfying $\|v_x\|=1$ and $v^T_x\nabla c_i(x)\leq-\delta_c$ for all $i\in\cA(x;\theta)$, where $\mcT_{\mcX}(x)$ is the tangent cone of $\mcX$ at $x$, and
\beq\label{def-cAcS}
\cF(\theta)=\{x\in\mcX\big|\|[c(x)]_+\|\leq\theta\},\quad\cA(x;\theta)=\{i|c_i(x)\geq-\theta,\ 1\leq i\leq \tn\}.
\eeq
\item[(ii)] For each $x\in\mcX$, there exists some $\hat y_x\in\mcY$ such that $d_i(x,\hat y_x)<0$ for all $i=1,2,\dots,\tm$, and moreover, $\delta_d:=\inf\{-d_i(x,\hat y_x)|x\in\mcX,\ i=1,2,\dots,\tm\}>0$.
\item[(iii)] $F(\cdot,y)$ is $L_F$-Lipschitz continuous on $\mcX$ for any $y\in\mcY$.
\ei
\end{assumption}

\begin{rem}
\bi
\item[(i)] Assumption \ref{mfcq}(i) can be viewed as a robust counterpart of MFCQ.  It implies that MFCQ holds for all the minimization problems, resulting from the minimization part of \eqref{prob} by fixing  $y\in\mcY$ and perturbing $c_i(x)$ at most by $\theta$.
\item[(ii)] The latter part of Assumption \ref{mfcq}(ii) can be weakened to the one that the pointwise Slater's condition holds for the constraint on $y$ in \eqref{prob}, that is, there exists $\hat y_x\in\mcY$ such that $d(x,\hat y_x)<0$ for each $x\in\mcX$. Indeed, if $\delta_d>0$, Assumption \ref{mfcq}(ii) holds. Otherwise, one can solve the perturbed counterpart of \eqref{prob} with $d(x,y)$ being replaced by $d(x,y)-\epsilon$ for some suitable $\epsilon>0$ instead, which satisfies Assumption \ref{mfcq}(ii).
\item[(iii)] In view of Assumption \ref{a1},  one can observe that  if $p$ is Lipschitz continuous on $\mcX$, $F(\cdot,y)$ is Lipschitz continuous on $\mcX$ for any $y\in\mcY$.  Thus, Assumption \ref{mfcq}(iii) is mild.
\ei
\end{rem}

In addition, to characterize the approximate solution found by Algorithm \ref{AL-alg}, we review a notion so-called an $\varepsilon$-KKT solution of problem \eqref{prob}, which was introduced in \cite[Definition 2]{lu2024first}.

\begin{defi} \label{approx-kkt-pt}
For any $\varepsilon>0$, $(x,y)$ is said to be an $\varepsilon$-KKT point of problem \eqref{prob} if there exists $(\lambda_\bfx,\lambda_\bfy)\in\bR^{\tn}_+\times\bR^{\tm}_+$ such that
\begin{align*}
& \dist(0, \partial_x F(x,y)+\nabla c(x)\lambda_\bfx-\nabla_x d(x,y) \lambda_\bfy) \leq \varepsilon,   \\
& \dist(0, \partial_y F(x,y)-\nabla_y d(x,y) \lambda_\bfy) \leq \varepsilon,   \\
& \|[c(x)]_+\| \leq  \varepsilon, \quad |\langle \lambda_\bfx, c(x) \rangle|\leq \varepsilon,  \\
& \|[d(x,y)]_+\| \leq  \varepsilon, \quad |\langle \lambda_\bfy, d(x,y) \rangle| \leq \varepsilon. 
\end{align*}
\end{defi} 

Recall that $\mcX=\dom\,p$ and $\mcY=\dom\,q$. To study complexity of Algorithm \ref{AL-alg}, we define
\begin{align}
&f^*(x):=\max\{F(x,y)|d(x,y)\leq0\},\label{fstarx}\\
&F_{\rm hi}:=\max\{F(x,y)|(x,y)\in\mcX\times\mcY\},\quad F_{\rm low}:=\min\{F(x,y)|(x,y)\in\mcX\times\mcY\},\label{Fhi}\\
&\Delta:=F_{\rm hi}-F_{\rm low}, \quad  r:=2\delta_d^{-1}\Delta,\label{def-r} \\
&K:=\left\lceil\log\varepsilon/\log\tau\right\rceil_+, \quad \bbK:=\{0,1,\ldots, K+1\}, \label{K1} 
\end{align}
where $\delta_d$ is given in Assumption \ref{mfcq}, and $\varepsilon$ and $\tau$ are some input parameters of Algorithm \ref{AL-alg}. For convenience, we define $\bbK-1=\{k-1| k\in\bbK\}$. One can observe from Assumption \ref{a1} that $F_{\rm hi}$ and $F_{\rm low}$ are finite. Besides, one can easily observe that 
\beq \label{F-gap}
f^*(x)\geq F_{\rm low}, \   F(x,y)-f^*(x)\leq \Delta \quad \forall x\in\mcX, y\in\mcY.
\eeq

We are now ready to present an \emph{iteration and operation complexity} of Algorithm~\ref{AL-alg} for finding an $\cO(\varepsilon)$-KKT solution of problem \eqref{prob}, whose proof is deferred to Section \ref{sec:proofs}.

\begin{thm}\label{complexity}
Suppose that Assumptions \ref{a1}, \ref{knownfeas} and \ref{mfcq} hold. Let $\{(x^k,y^k,\lambda^k_\bfx,\lambda^k_\bfy)\}_{k\in\bbK}$ be generated by Algorithm \ref{AL-alg}, $D_\bfx$, $D_\bfy$, $c_{\rm hi}$, $d_{\rm hi}$, $\Delta$ and $K$ be defined in  \eqref{mmax-D}, \eqref{cdhi}, \eqref{def-r} and \eqref{K1}, $L_F$, $L_{\nabla f}$, $L_{\nabla d}$, $L_{\nabla c}$, $L_c$, $L_{\nabla d}$, $L_d$, $\delta_c$, $\delta_d$ and $\theta$  be given in Assumptions \ref{a1} and \ref{mfcq}, $\varepsilon$, $\tau$, $\Lambda$ and $\lambda_\bfy^0$ be given in Algorithm \ref{AL-alg}, and
\begin{align}
&L= L_{\nabla f}+L_c^2+c_{\rm hi}L_{\nabla c}+\Lambda L_{\nabla c}+L_d^2+d_{\rm hi}L_{\nabla d}+L_{\nabla d}\sqrt{\|\lambda_\bfy^0\|^2+\frac{2(\Delta+D_\bfy)}{1-\tau}},\label{hL}\\
&\alpha=\min\left\{1, \sqrt{8\sigma/L}\right\},\quad \delta= (2+\alpha^{-1})L D_\bfx^2+\max\{2\sigma,L/4\}D_\bfy^2,\label{ho}\\
&M=16\max\left\{1/(2L_c^2),4/(\alpha L_c^2)\right\}\left[81/\min\{L_c^2,\sigma\}+ 3L\right]^2\nn\\
&\ \ \ \ \ \ \times\left(\delta+2\alpha^{-1}\Big(\Delta+\frac{\Lambda^2}{2}+\frac{3}{2}\|\lambda_\bfy^0\|^2+\frac{3(\Delta+D_\bfy)}{1-\tau}+\rho_kd_{\rm hi}^2+L D_\bfx^2\Big)\right),\label{hM}\\
&T= \Bigg\lceil16\left(2\Delta+\Lambda+\frac{1}{2}(\tau^{-1}+\|\lambda_\bfy^0\|^2)+\frac{\Delta+D_\bfy}{1-\tau}+\frac{\Lambda^2}{2}\right) L+8(1+\sigma^{-2}L^2)\Bigg\rceil_+, \label{hT}\\
&\tlambda^{K+1}_\bfx = [\lambda^K_\bfx+c(x^{K+1})/\tau^K]_+.\label{tlx}
\end{align}
Suppose that 
\begin{align}
\varepsilon^{-1} \geq \max\Bigg\{&1, \theta^{-1}\Lambda, \theta^{-2}\Big\{4\Delta+2\Lambda+\tau^{-1}+\|\lambda_\bfy^0\|^2+\frac{2(\Delta+D_\bfy)}{1-\tau} \nn \\
& +L_c^{-2} +\sigma^{-2}L+\Lambda^2\Big\}, \frac{4\|\lambda_\bfy^0\|^2}{\delta_d^2\tau}+\frac{8(\Delta+D_\bfy)}{\delta_d^2\tau(1-\tau)}\Bigg\}. \label{cond}
\end{align}
Then the following statements hold.
\begin{enumerate}[label=(\roman*)]
\item Algorithm \ref{AL-alg} terminates after $K+1$ outer iterations and outputs an approximate stationary point $(x^{K+1},y^{K+1})$ of \eqref{prob} satisfying
\begin{align}
&\dist(0,\partial_x F(x^{K+1},y^{K+1})+\nabla c(x^{K+1})\tl_x^{K+1}-\nabla_xd(x^{K+1},y^{K+1})\lambda^{K+1}_\bfy) \leq \varepsilon, \label{t1-1} \\
& \dist\left(0,\partial_y F(x^{K+1},y^{K+1})-\nabla_y d(x^{K+1},y^{K+1})\lambda^{K+1}_\bfy\right)\leq\varepsilon, \label{t1-2} \\
&\|[c(x^{K+1})]_+\|\leq \varepsilon\delta_c^{-1}\left(L_F +2L_d\delta_d^{-1}(\Delta+D_\bfy)+1\right), \label{t1-3} \\
& |\langle\tl^{K+1}_\bfx,c(x^{K+1})\rangle| \leq \varepsilon \delta_c^{-1}(L_F +2L_d\delta_d^{-1}(\Delta+D_\bfy)+1)\nn\\
&\qquad\qquad\qquad\qquad\,\,
\times\max\{\delta_c^{-1}(L_F +2L_d\delta_d^{-1}(\Delta+D_\bfy)+1), \Lambda\}, \label{t1-4} \\
&\|[d(x^{K+1},y^{K+1})]_+\|\leq2\varepsilon\delta_d^{-1}(\Delta+D_\bfy),\label{t1-5}\\
& |\langle \lambda^{K+1}_\bfy, d(x^{K+1},y^{K+1})\rangle| \leq 2\varepsilon\delta_d^{-1}(\Delta+D_\bfy)\max\{2\delta_d^{-1} (\Delta+D_\bfy), \|\lambda_\bfy^0\|\}. \label{t1-6}
\end{align} 
\item The total number of evaluations of $\nabla f$, $\nabla c$, $\nabla d$ and proximal operators of $p$ and $q$ performed in Algorithm \ref{AL-alg} is at most $N$, respectively, where
\begin{align}
N=&\ 3397\max\left\{2,\sqrt{L/(2\sigma)}\right\}T(1-\tau^{7/2})^{-1}\nn\\
&\ \times(\tau\varepsilon)^{-7/2}\left(20K\log(1/\tau)+2(\log M)_++2+2\log(2T) \right).\label{N2}
\end{align}
\end{enumerate}
\end{thm}

\begin{rem}
\bi
\item[(i)] The condition \eqref{cond} on $\varepsilon$ is to ensure that the final penalty parameter $\rho_K$ in Algorithm \ref{AL-alg} is large enough so that  feasibility and complementarity slackness are nearly satisfied at $(x^{K+1},y^{K+1},\tl^{K+1}_\bfx, \lambda^{K+1}_\bfy)$. 
\item[(ii)]
One can observe from Theorem \ref{complexity} that Algorithm \ref{AL-alg} enjoys an iteration complexity of $\cO(\log\varepsilon^{-1})$ and an operation complexity of $\cO(\varepsilon^{-3.5}\log\varepsilon^{-1})$, measured by the amount of evaluations of $\nabla f$, $\nabla c$, $\nabla d$ and proximal operators of $p$ and $q$, for finding an $\cO(\varepsilon)$-KKT solution $(x^{K+1},y^{K+1})$ of \eqref{prob} such that
\begin{align*}
& \dist\left(\partial_x F(x^{K+1},y^{K+1})+\nabla c(x^{K+1})\tl_\bfx-\nabla_x d(x^{K+1},y^{K+1}) \lambda_\bfy^{K+1}\right) \leq \varepsilon,\\
& \dist\left(\partial_y F(x^{K+1},y^{K+1})-\nabla_y d(x^{K+1},y^{K+1}) \lambda_\bfy^{K+1}\right)\leq \varepsilon,\\
& \|[c(x^{K+1})]_+\|=\cO(\varepsilon), \quad |\langle \tl_\bfx^{K+1}, c(x^{K+1}) \rangle|=\cO(\varepsilon),\\
& \|[d(x^{K+1},y^{K+1})]_+\|=\cO(\varepsilon), \quad |\langle \lambda_\bfy^{K+1}, d(x^{K+1},y^{K+1}) \rangle|=\cO(\varepsilon),
\end{align*}
where $\tl_\bfx^{K+1}\in\bR_+^{\tn}$ is defined in \eqref{tlx} and $\lambda_\bfy^{K+1}\in\bR_+^{\tm}$ is given in Algorithm \ref{AL-alg}.
\item[(iii)] It shall be mentioned that an $\cO(\varepsilon)$-KKT solution of \eqref{prob} can be found by \cite[Algorithm 3]{lu2024first} with an operation complexity of $\cO(\varepsilon^{-4}\log\varepsilon^{-1})$ (see \cite[Theorem 3]{lu2024first}). As a result, the operation complexity of Algorithm \ref{AL-alg} improves that of \cite[Algorithm 3]{lu2024first} by a factor of $\epsilon^{-1/2}$.
\ei
\end{rem}

\section{Numerical results}\label{sec:exp}
In this section, we conduct some preliminary experiments to test the performance of our proposed method (namely, Algorithms \ref{mmax-alg2} and \ref{AL-alg}), and compare them with an alternating gradient projection method (AGP) \cite[Algorithm 1]{xu2023unified} and an augmented Lagrangian method (ALM) \cite[Algorithm 3]{lu2024first}, respectively. All the algorithms are coded in Matlab, and all the computations are performed on a laptop with a 2.30 GHz Intel i9-9880H 8-core processor and 16 GB of RAM.

\subsection{Unconstrained nonconvex-strongly-concave minimax optimization with quadratic objective}
In this subsection, we consider the problem
\begin{align}\label{prob-exp0}
\min_{x}\max_{y}\  x^TAx+x^TBy+y^TCy+c^Tx+d^Ty+\cI_{[-1,1]^n}(x)-\cI_{[-1,1]^m}(y),
\end{align}
where $A\in\bR^{n\times n}$, $B\in\bR^{n\times m}$, $C\in\bR^{m\times m}$, $c\in\bR^{n}$, $d\in\bR^{m}$, and $\cI_{[-1,1]^n}(\cdot)$ and $\cI_{[-1,1]^m}(\cdot)$ are the indicator functions of $[-1,1]^n$ and $[-1,1]^m$ respectively.

For each pair $(n,m)$, we randomly generate $10$ instances of problem \eqref{prob-exp0}. Specifically, we construct $A=UDU^T$,  where $U=\mathrm{orth}(\mathrm{randn}(n))$, and $D$  is a diagonal matrix with entries independently drawn from a normal distribution with mean $0$ and standard deviation $0.1$. Matrix $C$ is generated in a similar manner, except the diagonal entries of the corresponding matrix are drawn independently from a uniform distribution over $[2,3]$.  In addition, we randomly generate vectors $c$ and $d$ with all the entries independently drawn from a normal distribution with mean $0$ and standard deviation $0.1$. 

Notice that \eqref{prob-exp0} is a special case of \eqref{mmax-prob} with $h(x,y)=x^TAx+x^TBy+y^TCy+c^Tx+d^Ty$, $p(x)=\cI_{[-1,1]^n}(x)$, and $q(y)=\cI_{[-1,1]^m}(y)$ and can be suitably solved by Algorithm \ref{mmax-alg2} and AGP \cite[Algorithm 1]{xu2023unified}. In addition, problem \eqref{prob-exp0} is equivalent to as the following minimization problem
\begin{align}
\min_{x} \Phi(x),\label{prob-hyper0}
\end{align}
where $\Phi$ is the hyper-objective function defined as
\[
\Phi(x)=\max_{y} x^TAx+x^TBy+y^TCy+c^Tx+d^Ty+\cI_{[-1,1]^n}(x)-\cI_{[-1,1]^m}(y).
\]

For Algorithm~\ref{mmax-alg2}, we set the parameters to 
$(\epsilon,\hat\epsilon_0)=(10^{-2},5\times10^{-3})$. For AGP, we use the parameter settings as specified in \cite[Subsection 3.1]{xu2023unified}. Both algorithms are initialized with the all-one vector. Each algorithm is terminated once a $10^{-2}$-primal-dual stationary point $(x^k,y^k)$ of \eqref{prob-exp0} is found for some $k$, and the pair $(x^k,y^k)$ is returned as an approximate solution to \eqref{prob-exp0}.

The computational results of the aforementioned algorithms on the randomly generated instances are presented in Table~\ref{t0}. Specifically, the values of $n$ and $m$ are listed in the first two columns. For each pair $(n, m)$, the average initial hyper-objective value $\Phi(x^0)$, the average final hyper-objective value $\Phi(x^k)$, and the average CPU time (in seconds) over 10 random instances are reported in the remaining columns. It can be observed that both Algorithm~\ref{mmax-alg2} and AGP~\cite[Algorithm 1]{xu2023unified} yield approximate solutions with comparable hyper-objective values, which are significantly lower than the initial value. However, Algorithm~\ref{mmax-alg2} consistently achieves significantly lower CPU times, which may be attributed to its more favorable dependence on condition numbers.

\begin{table}[H]
\centering
\resizebox{0.95\linewidth}{!}{
\begin{tabular}{cc||c||ll||ll}
\hline
&&Initial hyper-objective value&\multicolumn{2}{c||}{Final hyper-objective value}&\multicolumn{2}{c}{CPU time (seconds)}\\
$n$&$m$&& Algorithm \ref{mmax-alg2} &AGP& Algorithm \ref{mmax-alg2} &AGP\\\hline
50&50&$4.30$&$-0.30$&$-0.29$&19.3&100.0\\
100&100&$10.34$&$-1.13$&$-1.10$&82.6&428.6\\
150&150&$22.16$&$-1.01$&$-1.09$&176.5&910.3\\
200&200&$32.52$&$-1.43$&$-1.39$&222.6&1141.1\\
250&250&$69.19$&$-1.80$&$-1.83$&312.7&1219.1\\
300&300&$108.76$&$-2.11$&$-2.07$&400.2&1245.5\\
350&350&$124.88$&$-2.06$&$-2.09$&483.0&1366.9\\
400&400&$175.78$&$-2.17$&$-2.13$&512.9&1443.3\\
\hline
\end{tabular}
}
\caption{Numerical results for problem \eqref{prob-exp0}}\label{t0}
\end{table}

\subsection{Constrained nonconvex-strongly-concave minimax optimization with quadratic objective and linear constraints}

In this subsection, we consider the problem
\begin{align}\label{prob-exp}
\min_{\widehat Ax\leq \hat b}\max_{\widetilde Ax+\widetilde B y\leq \tilde b} x^TAx+x^TBy+y^TCy+c^Tx+d^Ty+\cI_{[-1,1]^n}(x)-\cI_{[-1,1]^m}(y),
\end{align}
where $A\in\bR^{n\times n}$, $B\in\bR^{n\times m}$, $C\in\bR^{m\times m}$, $c\in\bR^{n}$, $d\in\bR^{m}$, $\widehat A\in\bR^{\tilde n\times n}$, $\hat b\in\bR^{\tilde n}$, $\widetilde A\in\bR^{\tilde m\times n}$, $\widetilde B\in\bR^{\tilde m\times m}$, $\tilde b\in\bR^{\tilde m}$, and $\cI_{[-1,1]^n}(\cdot)$ and $\cI_{[-1,1]^m}(\cdot)$ are the indicator functions of $[-1,1]^n$ and $[-1,1]^m$ respectively.

For each tuple $(n,m,\tilde n,\tilde m)$, we randomly generate $10$ instances of problem \eqref{prob-exp}. Specifically, we construct $A=UDU^T$,  where $U=\mathrm{orth}(\mathrm{randn}(n))$, and $D$  is a diagonal matrix with entries independently drawn from a normal distribution with mean $0$ and standard deviation $0.1$. Matrix $C$ is generated in a similar manner, except its diagonal entries are independently drawn from a uniform distribution over $[10,11]$.  In addition, we randomly generate matrices $B$, $\widehat A$, $\widetilde A$, $\widetilde B$, and vectors $c$, $d$, $\tilde b$ with all the entries independently drawn from a normal distribution with mean $0$ and standard deviation $0.1$. Finally, we randomly generate $x_{\rm\bf nf}\in[-1,1]^n$ by first sampling each entry independently from a normal distribution with mean $0$ and standard deviation $0.1$, then projecting the resulting vector onto $[-1,1]^n$. We choose $\hat b$ such that $x_{\rm\bf nf}$ is $0.1$-nearly feasible (see Assumption \ref{knownfeas}) for problem \eqref{prob-exp}. 

Notice that \eqref{prob-exp} is a special case of \eqref{prob} with
\begin{align*}
&f(x,y)=x^TAx+x^TBy+y^TCy+c^Tx+d^Ty,\quad p(x)=\cI_{[-1,1]^n}(x), \\
&q(y)=\cI_{[-1,1]^m}(y),\quad c(x)=\widehat Ax-\hat b,\quad d(x,y)=\widetilde Ax+\widetilde B y-\tilde b,
\end{align*}
and can be suitably solved by Algorithm \ref{AL-alg} and ALM \cite[Algorithm 3]{lu2024first}. In addition, problem \eqref{prob-exp} is equivalent to the following minimization problem
\begin{align}
\min_{\widehat Ax\leq\hat b} \Phi(x),\label{prob-hyper}
\end{align}
where $\Phi$ is the hyper-objective function defined as
\[
\Phi(x)=\max_{\widetilde Ax+\widetilde B y\leq \tilde b} x^TAx+x^TBy+y^TCy+c^Tx+d^Ty+\cI_{[-1,1]^n}(x)-\cI_{[-1,1]^m}(y).
\]

We choose the parameters as $(\varepsilon,\tau,\Lambda)=(10^{-2},0.5,10)$ for both Algorithm \ref{AL-alg} and ALM \cite[Algorithm 3]{lu2024first}, and initialize them at zero. The algorithms are terminated once a $10^{-2}$-relative-KKT point\footnote{We say $(x,y)$ is an $\epsilon$-relative-KKT point of \eqref{prob-exp} if it is an $(|\Phi(x)|+1)\epsilon$-KKT point of \eqref{prob-exp}.} $(x_k,y_k)$ of \eqref{prob-exp} is found for some $k$, and we output $(x_k,y_k)$ as an approximate solution to \eqref{prob-exp}.

The computational results of the aforementioned algorithms for the instances randomly generated above are presented in Table \ref{t1}. Specifically, the values of $n$, $m$, $\tilde n$, and $\tilde m$ are listed in the first four columns. For each tuple $(n,m,\tilde n,\tilde m)$, the average initial hyper-objective value $\Phi(x^0)$, the average final hyper-objective value $\Phi(x^k)$, and the average CPU time (in seconds) over $10$ random instances are given in the rest of the columns. We observe that both Algorithm~\ref{AL-alg} and ALM~\cite[Algorithm 3]{lu2024first} produce approximate solutions with comparable hyper-objective values that are significantly lower than the initial ones. Moreover, Algorithm~\ref{AL-alg} consistently achieves substantially lower CPU times since it effectively exploits the strong concavity structure of the problem.

\begin{table}[H]
\centering
\resizebox{0.95\linewidth}{!}{
\begin{tabular}{cccc||c||ll||ll}
\hline
&&&&Initial hyper-objective value&\multicolumn{2}{c||}{Final hyper-objective value}&\multicolumn{2}{c}{CPU time (seconds)}\\
$n$&$m$&$\tn$&$\tm$&&Algorithm \ref{AL-alg}&ALM&Algorithm \ref{AL-alg}&ALM\\\hline
50&100&5&10&$-0.52$&$-183.09$&$-183.18$&332.8&1111.9\\
100&200&10&20&$-0.40$&$-625.04$&$-625.76$&2001.9&2996.1\\
150&300&15&30&$-0.45$&$-895.71$&$-895.02$&4535.1&6396.9\\
200&400&20&40&$-0.34$&$-1255.49$&$-1254.74$&6252.2&9653.4\\
250&500&25&50&$-0.45$&$-1631.83$&$-1632.54$&8343.8&13522.1\\
\hline
\end{tabular}
}
\caption{Numerical results for problem \eqref{prob-exp}}\label{t1}
\end{table}

\section{Proof of the main result} \label{sec:proofs}
In this section we provide a proof of our main results presented in Sections \ref{minimax} and \ref{sec:main}, which are particularly Theorems  \ref{mmax-thm} and \ref{complexity}.

\subsection{Proof of the main results in Section~\ref{minimax}}\label{sec:proof2-2}
In this subsection we prove Theorem \ref{mmax-thm}. Before proceeding, let $\{(x^k,y^k)\}_{k\in \bbT}$ denote all the iterates generated by Algorithm~\ref{mmax-alg2}, where $\bbT$ is a subset of consecutive nonnegative integers starting from $0$. Also, we define $\bbT-1 = \{k-1: k \in \bbT\}$. We first establish two lemmas and then use them to prove Theorem \ref{mmax-thm} subsequently.

The following lemma shows that an approximate primal-dual stationary point  of \eqref{ppa-subprob} is found at each iteration of Algorithm~\ref{mmax-alg2}, and also provides an estimate of operation complexity for finding it.

\begin{lemma}\label{innercplx-alg6}
Suppose that Assumption~\ref{mmax-a} holds. Let $\{(x^k,y^k)\}_{k\in\bbT}$ be generated by Algorithm~\ref{mmax-alg2}, $\bH^*$, $D_\bfx$, $D_\bfy$, $\bH_{\rm low}$, $\halpha$, $\hdelta$ be defined in \eqref{mmax-prob}, \eqref{mmax-D}, \eqref{mmax-bnd}, \eqref{mmax-balpha} and \eqref{mmax-tP}, $L_{\nabla\bh}$ be given in Assumption \ref{mmax-a}, $\epsilon$, $\hat\epsilon_k$ be given in Algorithm~\ref{mmax-alg2}, and
\begin{align}
&\ \hat N_k:=3397\Bigg\lceil\max\left\{2,\sqrt{\frac{L_{\nabla \bh}}{2\sigma_y}}\right\}\log\frac{4\max\left\{\frac{1}{2L_{\nabla \bh}},\min\left\{\frac{1}{2\sigma_y},\frac{4}{\halpha L_{\nabla \bh}}\right\}\right\}\left(\hdelta+2\halpha^{-1}(\bH^*-\bH_{\rm low}+L_{\nabla \bh} D_\bfx^2)\right)}{\left[9L_{\nabla \bh}^2/\min\{L_{\nabla \bh},\sigma_y\}+ 3L_{\nabla \bh}\right]^{-2}\hat\epsilon_k^2}\Bigg\rceil_+.\label{mmax-Nk}
\end{align}
Then for all $0\leq k\in\bbT-1$, $(x^{k+1},y^{k+1})$ is an $\hat\epsilon_k$-primal-dual stationary point of \eqref{ppa-subprob}. Moreover, the total number of evaluations of $\nabla \bh$ and proximal operators of $p$ and $q$ performed at iteration $k$ of Algorithm~\ref{mmax-alg2} for generating $(x^{k+1},y^{k+1})$ is no more than $\hat N_k$, respectively. 
\end{lemma}

\begin{proof}
Let $(x^*,y^*)$ be an optimal solution of \eqref{mmax-prob}. Recall that $H$, $H_k$ and $h_k$ are respectively given in \eqref{mmax-prob}, \eqref{ppa-subprob} and \eqref{mmax-sub}, $\mcX=\dom\,p$ and $\mcY=\dom\,q$. Notice that $x^*, x^k\in\mcX$. Then we have
\begin{align}
\bH_{k,*}:=\min_x\max_y H_k(x,y) =&\ \min_x\max_y\left \{\bH(x,y)+L_{\nabla \bh}\|x-x^k\|^2\right\}\notag\\
\leq&\ \max_y \{\bH(x^*,y)+L_{\nabla \bh}\|x^*-x^k\|^2\} \overset{\eqref{mmax-prob}\eqref{mmax-D}}\leq \bH^*+L_{\nabla \bh} D_\bfx^2.\label{Hkstar}
\end{align}
 Moreover, by $\mcX=\dom\,p$, $\mcY=\dom\,q$, \eqref{mmax-D} and \eqref{mmax-bnd}, one has
\begin{align}
\bH_{k, \rm low}:=\min_{(x,y)\in\dom\,p\times\dom\,q} H_k(x,y) =\ \min_{(x,y)\in\mcX\times\mcY}\left\{\bH(x,y)+L_{\nabla \bh}\|x-x^k\|^2\right\}\overset{\eqref{mmax-bnd}}\geq\bH_{\rm low}.\label{hklow}
\end{align}
In addition, by Assumption~\ref{mmax-a} and the definition of $\bh_k$ in \eqref{mmax-sub}, it is not hard to verify that $\bh_k(x,y)$ is $L_{\nabla \bh}$-strongly-convex in $x$, $\sigma_y$-strongly-concave in $y$, and $3L_{\nabla \bh}$-smooth on its domain. Also, recall from Remark \ref{ppa-rem} that $(x^{k+1},y^{k+1})$  results from applying Algorithm~\ref{mmax-alg1} to problem \eqref{ppa-subprob}.  
The conclusion of this lemma then follows by using \eqref{Hkstar} and \eqref{hklow} and applying Theorem~\ref{ea-prop} to  \eqref{ppa-subprob} with $\bar\epsilon=\hat\epsilon_k$, $\sigma_x=L_{\nabla \bh}$, $\sigma_y=\sigma$, $L_{\nabla \h}=3L_{\nabla \bh}$, $\bar \alpha= \halpha$, $\bar \delta=\hdelta$, $\H_{\rm low}=H_{k, \rm low}$, and $\H^*=H_{k,*}$.
\end{proof}

The following lemma provides an upper bound on the least progress of the solution sequence of Algorithm~\ref{mmax-alg2} and also on the last-iterate objective value of \eqref{mmax-prob}.

\begin{lemma} \label{output-prop}
Suppose that Assumption~\ref{mmax-a} holds. Let $\{x^k\}_{k\in\bbT}$ be generated by Algorithm~\ref{mmax-alg2}, $H$, $\bH^*$ and $D_\bfy$ be defined in \eqref{mmax-prob} and \eqref{mmax-D}, $L_{\nabla\bh}$ be given in Assumption \ref{mmax-a}, and $\epsilon$, $\hat\epsilon_0$ and $\hat x^0$ be given in Algorithm~\ref{mmax-alg2}. Then for all $0\leq K\in\bbT-1$, we have
\begin{align}
&\min_{0\leq k\leq K}\|x^{k+1}-x^k\|\leq\frac{\max_y\bH(\hat x^0,y)-\bH^*}{L_{\nabla \bh}(K+1)}+\frac{2\hat\epsilon_0^{2}(1+\sigma_y^{-2}L_{\nabla \bh}^2)}{L_{\nabla \bh}^2(K+1)},\label{mmax-xbnd}\\
&\max_y\bH(x^{K+1},y)\leq \max_y\bH(\hat x^0,y)+2\hat\epsilon_0^2\left(L_{\nabla \bh}^{-1}+\sigma_y^{-2}L_{\nabla \bh}\right).\label{K-upperbnd}
\end{align}
\end{lemma}

\begin{proof}
For convenience of the proof, let
\begin{align}
&\bH^*(x)=\max\limits_y\bH(x,y),\label{mmax-tg}\\
&\bH^*_k(x)=\max\limits_y \bH_k(x,y),\quad y^{k+1}_*=\argmax_y\bH_k(x^{k+1},y). \label{mmax-ystar}
\end{align}
One can observe from these, \eqref{ppa-subprob} and \eqref{mmax-sub} that
\beq \label{mmax-gk}
H^*_k(x)=\bH^*(x)+L_{\nabla \bh}\|x-x^k\|^2.  
\eeq
By this and Assumption~\ref{mmax-a}, one can also see that $H^*_k$ is $L_{\nabla \bh}$-strongly convex on $\dom\,p$.
In addition, recall from Lemma \ref{innercplx-alg6} that $(x^{k+1},y^{k+1})$ is an $\hat\epsilon_k$-primal-dual stationary point of problem \eqref{ppa-subprob} for all $0\leq k\in\bbT-1$. It then follows from Definition~\ref{def2} that there exist some $u\in\partial_x \bH_k(x^{k+1},y^{k+1})$ and $v\in\partial_y \bH_k(x^{k+1},y^{k+1})$ with $\|u\|\leq\hat\epsilon_k$ and $\|v\|\leq\hat\epsilon_k$. Also, by \eqref{mmax-ystar}, one has $0\in\partial_y \bH_k(x^{k+1},y^{k+1}_*)$, which, together with $v\in\partial_y \bH_k(x^{k+1},y^{k+1})$ and  $\sigma_y$-strong concavity of $\bH_k(x^{k+1},\cdot)$, implies that
$\langle -v,y^{k+1}-y^{k+1}_*\rangle\geq \sigma_y\|y^{k+1}-y^{k+1}_*\|^2$. 
This and $\|v\|\leq\hat\epsilon_k$ yield  
\begin{equation}\label{mmax-e1}
\|y^{k+1}-y^{k+1}_*\|\leq\sigma_y^{-1}\hat\epsilon_k.
\end{equation}
In addition, by $u\in\partial_x \bH_k(x^{k+1},y^{k+1})$, \eqref{ppa-subprob} and \eqref{mmax-sub}, one has \begin{equation}\label{mmax-e2}
u\in\nabla_x\bh(x^{k+1},y^{k+1})+\partial p(x^{k+1})+2L_{\nabla \bh}(x^{k+1}-x^k).
\end{equation}
Also, observe from \eqref{ppa-subprob}, \eqref{mmax-sub} and \eqref{mmax-ystar} that
\begin{equation*}
\partial H^*_k(x^{k+1})=\nabla_x\bh(x^{k+1},y^{k+1}_*)+\partial p(x^{k+1})+2L_{\nabla \bh}(x^{k+1}-x^k),
\end{equation*}
which together with \eqref{mmax-e2} yields
\begin{equation*}
u+\nabla_x\bh(x^{k+1},y^{k+1}_*)-\nabla_x\bh(x^{k+1},y^{k+1})\in\partial H^*_k(x^{k+1}).
\end{equation*}
By this and $L_{\nabla \bh}$-strong convexity of $H^*_k$, one has
\begin{equation}\label{strongcvx}
H^*_k(x^k)\geq H^*_k(x^{k+1})+\langle u+\nabla_x\bh(x^{k+1},y^{k+1}_*)-\nabla_x\bh(x^{k+1},y^{k+1}),x^k-x^{k+1}\rangle+L_{\nabla \bh}\|x^k-x^{k+1}\|^2/2.
\end{equation}
Using this, \eqref{mmax-gk}, \eqref{mmax-e1}, \eqref{strongcvx}, $\|u\|\leq\hat\epsilon_k$, and the Lipschitz continuity of $\nabla \bh$, we obtain
\begin{align*}
&\bH^*(x^k)-\bH^*(x^{k+1})\overset{\eqref{mmax-gk}}{=}H^*_k(x^k)- H^*_k(x^{k+1})+L_{\nabla\bh}\|x^k-x^{k+1}\|^2\\
&\overset{\eqref{strongcvx}}{\geq}\langle u+\nabla_x\bh(x^{k+1},y^{k+1}_*)-\nabla_x\bh(x^{k+1},y^{k+1}),x^k-x^{k+1}\rangle+3L_{\nabla \bh}\|x^k-x^{k+1}\|^2/2\\
&\ \geq \big(-\|u+\nabla_x\bh(x^{k+1},y^{k+1}_*)-\nabla_x\bh(x^{k+1},y^{k+1})\|\|x^k-x^{k+1}\|+L_{\nabla \bh}\|x^k-x^{k+1}\|^2/2\big) +L_{\nabla \bh}\|x^k-x^{k+1}\|^2 \\
&\ \geq -(2L_{\nabla \bh})^{-1}\|u+\nabla_x\bh(x^{k+1},y^{k+1}_*)-\nabla_x\bh(x^{k+1},y^{k+1})\|^2+L_{\nabla \bh}\|x^k-x^{k+1}\|^2 \\
&\ \geq-L_{\nabla \bh}^{-1}\|u\|^2-L_{\nabla \bh}^{-1}\|\nabla_x \bh(x^{k+1},y^{k+1}_*)-\nabla_x \bh(x^{k+1},y^{k+1})\|^2+L_{\nabla \bh}\|x^k-x^{k+1}\|^2\\
&\ \geq-L_{\nabla \bh}^{-1}\hat\epsilon_k^2-L_{\nabla\bh}\|y^{k+1}-y^{k+1}_*\|^2+L_{\nabla\bh}\|x^k-x^{k+1}\|^2\\
&\overset{\eqref{mmax-e1}}{\geq} -(L_{\nabla\bh}^{-1}+\sigma_y^{-2}L_{\nabla\bh})\hat\epsilon_k^2+L_{\nabla\bh}\|x^k-x^{k+1}\|^2, 
\end{align*}
where the second and fourth inequalities follow from Cauchy-Schwartz inequality, and the third inequality is due to Young's inequality, and the fifth inequality follows from $L_{\nabla \bh}$-Lipschitz continuity of $\nabla \bh$. Summing up the above inequality for $k=0,1,\dots, K$ yields
\begin{equation}\label{mmax-sum}
L_{\nabla \bh}\sum_{k=0}^K\|x^k-x^{k+1}\|^2\leq\bH^*(x^0)-\bH^*(x^{K+1})+(L_{\nabla\bh}^{-1}+\sigma_y^{-2}L_{\nabla \bh})\sum_{k=0}^K\hat\epsilon_k^2.
\end{equation}
In addition, it follows from \eqref{mmax-prob}, \eqref{mmax-D} and \eqref{mmax-tg} that
\begin{align}
&\bH^*(x^{K+1})=\max_y\bH(x^{K+1},y)\geq\min_x\max_y\bH(x,y)=\bH^*,\quad \bH^*(x^0)=\max_y\bH(x^0,y).\label{hhx0}
\end{align}
These together with \eqref{mmax-sum} yield
\begin{align*}
L_{\nabla \bh}(K+1)\min_{0\leq k\leq K}\|x^{k+1}-x^k\|^2\leq&\ L_{\nabla \bh}\sum_{k=0}^K\|x^k-x^{k+1}\|^2\\
\leq&\ \max_y\bH(x^0,y)-\bH^*+(L_{\nabla\bh}^{-1}+\sigma_y^{-2}L_{\nabla \bh})\sum_{k=0}^K\hat\epsilon_k^2,
\end{align*}
which, together with $x^0=\hat x^0$, $\hat\epsilon_k=\hat\epsilon_0(k+1)^{-1}$ and $\sum_{k=0}^K(k+1)^{-2}<2$, implies that \eqref{mmax-xbnd} holds.

Finally, we show that \eqref{K-upperbnd} holds. Indeed, it follows from \eqref{mmax-D}, \eqref{mmax-tg}, \eqref{mmax-sum}, \eqref{hhx0}, $\hat\epsilon_k=\hat\epsilon_0(k+1)^{-1}$, and $\sum_{k=0}^K(k+1)^{-2}<2$ that
\begin{align*}
\max_y\bH(x^{K+1},y) \overset{\eqref{mmax-tg}}{=}&\bH^*(x^{K+1}) 
\overset{\eqref{mmax-sum}}{\leq}\ \hh(x^0)+(L_{\nabla\bh}^{-1}+\sigma_y^{-2}L_{\nabla \bh})\sum_{k=0}^K\hat\epsilon_k^2\\
\overset{\eqref{hhx0}}{\leq}&\ \max_y\bH(x^0,y)+2\hat\epsilon_0^2(L_{\nabla\bh}^{-1}+\sigma_y^{-2}L_{\nabla \bh}).
\end{align*}
It then follows from this and $x^0=\hat x^0$ that \eqref{K-upperbnd} holds.
\end{proof}

We are now ready to prove Theorem~\ref{mmax-thm} using Lemmas \ref{innercplx-alg6} and \ref{output-prop}.

\begin{proof}[\textbf{Proof of Theorem~\ref{mmax-thm}}]
Suppose for contradiction that Algorithm~\ref{mmax-alg2} runs for more than $\wT+1$ outer iterations, where $\wT$ is given in \eqref{mmax-K}. By this and Algorithm~\ref{mmax-alg2}, one can then assert that \eqref{mmax-term} does not hold for all $0\leq k\leq \widehat T$. On the other hand, by \eqref{mmax-K} and \eqref{mmax-xbnd}, one has
\begin{align*}
\min_{0\leq k\leq \wT}\|x^{k+1}-x^k\|^2\ \overset{\eqref{mmax-xbnd}}{\leq}\ \frac{\max_y\bH(\hat x^0,y)-\bH^*}{L_{\nabla \bh}(\wT+1)}+\frac{2\hat\epsilon_0^{2}(1+\sigma_y^{-2}L_{\nabla \bh}^2)}{L_{\nabla \bh}^2(\wT+1)}\overset{\eqref{mmax-K}}{\leq}\frac{\epsilon^2}{16L_{\nabla \bh}^2},
\end{align*}
which implies that there exists some $0\leq k\leq \wT$ such that $\|x^{k+1}-x^k\|\leq\epsilon/(4L_{\nabla \bh})$,
and hence  \eqref{mmax-term} holds for such $k$, which contradicts the above assertion. Hence, Algorithm~\ref{mmax-alg2} must terminate in at most $\wT+1$ outer iterations.

Suppose that Algorithm~\ref{mmax-alg2} terminates at some iteration $0\leq k\leq \wT$, namely,   \eqref{mmax-term} holds for such $k$. We next show that its output $(\xe,\ye)=(x^{k+1},y^{k+1})$ is an $\epsilon$-primal-dual stationary point of \eqref{mmax-prob} and moreover it satisfies 
\eqref{upperbnd}. Indeed, recall from Lemma \ref{innercplx-alg6} that $(x^{k+1},y^{k+1})$ is an $\hat\epsilon_k$-primal-dual stationary point of 
\eqref{ppa-subprob}, namely, it satisfies $\dist(0,\partial_x \bH_k(x^{k+1},y^{k+1})) \leq \hat\epsilon_k$ and $\dist(0,\partial_y \bH_k(x^{k+1},y^{k+1})) \leq \hat\epsilon_k$. By these, \eqref{mmax-prob}, \eqref{ppa-subprob} and \eqref{mmax-sub},  there exists $(u,v)$ such that 
\begin{align*}
&u\in\partial_{x}\bH(x^{k+1},y^{k+1})+2L_{\nabla \bh}(x^{k+1}-x^k),\quad\|u\|\leq\hat\epsilon_k,\\
&v\in\partial_{y}\bH(x^{k+1},y^{k+1}),\quad\|v\|\leq\hat\epsilon_k.
\end{align*}
It then follows that $u-2L_{\nabla \bh}(x^{k+1}-x^k)\in\partial_{x}\bH(x^{k+1},y^{k+1})$ and $v\in\partial_{y}\bH(x^{k+1},y^{k+1})$. These together with \eqref{mmax-term}, \eqref{mmax-D},  and $\hat\epsilon_k\leq\hat\epsilon_0\leq\epsilon/2$ (see Algorithm~\ref{mmax-alg2}) imply that
\begin{align*}
&\dist\left(0,\partial_x\bH(x^{k+1},y^{k+1})\right)\leq\|u-2L_{\nabla \bh}(x^{k+1}-x^k)\|\leq\|u\|+2L_{\nabla \bh}\|x^{k+1}-x^k\|\overset{\eqref{mmax-term}}{\leq}\hat\epsilon_k+\epsilon/2\leq\epsilon,\\
&\dist\left(0,\partial_y\bH(x^{k+1},y^{k+1})\right)\leq\|v\|\leq\hat\epsilon_k<\epsilon.
\end{align*}
Hence, the output $(x^{k+1},y^{k+1})$ of Algorithm~\ref{mmax-alg2} is an $\epsilon$-primal-dual stationary point of \eqref{mmax-prob}. In addition, 
\eqref{upperbnd-old} holds due to Lemma \ref{output-prop}.

Recall from Lemma \ref{innercplx-alg6} that the number of evaluations of $\nabla \bh$ and proximal operators of $p$ and $q$ performed at iteration $k$ of Algorithm~\ref{mmax-alg2} is at most $\hat N_k$, respectively, where $\hat N_k$ is defined in \eqref{mmax-Nk}. Also, one can observe from the above proof and the definition of $\bbT$ that $|\bbT|\leq \wT+2$. It then follows that the total number of evaluations of $\nabla \bh$ and proximal operators of $p$ and $q$ in Algorithm~\ref{mmax-alg2} is respectively no more than $\sum_{k=0}^{|\bbT|-2}\hat N_k$. 
Consequently, to complete the rest of the proof of Theorem~\ref{mmax-thm}, it suffices to show that $\sum_{k=0}^{|\bbT|-2}\hat N_k \leq \widehat N$, where $\widehat N$ is given 
in \eqref{mmax-N-old}. Indeed, by \eqref{mmax-N-old}, \eqref{mmax-Nk} and $|\bbT|\leq \wT+2$, one has
\begin{align*}
&\sum_{k=0}^{|\bbT|-2}\hat N_k\overset{\eqref{mmax-Nk}}{\leq}\sum_{k=0}^{\wT}3397\times\Bigg\lceil\max\left\{2,\sqrt{\frac{L_{\nabla \bh}}{2\sigma_y}}\right\}\notag\\
&\ \ \ \ \times\log\frac{4\max\left\{\frac{1}{2L_{\nabla \bh}},\min\left\{\frac{1}{2\sigma_y},\frac{4}{\halpha L_{\nabla \bh}}\right\}\right\}\left( \hdelta+2\halpha^{-1}(\bH^*-\bH_{\rm low}+L_{\nabla \bh} D_\bfx^2)\right)}{\left[9L_{\nabla \bh}^2/\min\{L_{\nabla \bh},\sigma_y\}+ 3L_{\nabla \bh}\right]^{-2}\hat\epsilon_k^2}\Bigg\rceil_+\\
&\leq3397\times\max\left\{2,\sqrt{\frac{L_{\nabla \bh}}{2\sigma_y}}\right\}\notag\\
&\ \ \ \ \times\sum_{k=0}^{\wT}\left(\left(\log\frac{4\max\left\{\frac{1}{2L_{\nabla \bh}},\min\left\{\frac{1}{2\sigma_y},\frac{4}{\halpha L_{\nabla \bh}}\right\}\right\}\left( \hdelta+2\halpha^{-1}(\bH^*-\bH_{\rm low}+L_{\nabla \bh} D_\bfx^2)\right)}{\left[9L_{\nabla \bh}^2/\min\{L_{\nabla \bh},\sigma_y\}+ 3L_{\nabla \bh}\right]^{-2}\hat\epsilon_k^2}\right)_++1\right)\\
&\leq3397\times\max\left\{2,\sqrt{\frac{L_{\nabla \bh}}{2\sigma_y}}\right\}\notag\\
&\ \ \ \ \times\Bigg((\wT+1)\Bigg(\log\frac{4\max\left\{\frac{1}{2L_{\nabla \bh}},\min\left\{\frac{1}{2\sigma_y},\frac{4}{\halpha L_{\nabla \bh}}\right\}\right\}\left( \hdelta+2\halpha^{-1}(\bH^*-\bH_{\rm low}+L_{\nabla \bh} D_\bfx^2)\right)}{\left[9L_{\nabla \bh}^2/\min\{L_{\nabla \bh},\sigma_y\}+ 3L_{\nabla \bh}\right]^{-2}\hat\epsilon_0^2}\Bigg)_+\\
&\ \ \ \ +\wT+1+2\sum_{k=0}^{\wT}\log(k+1) \Bigg)\overset{\eqref{mmax-N-old}}{\leq} \widehat N,
\end{align*}
where the last inequality is due to \eqref{mmax-N-old} and $\sum_{k=0}^{\wT}\log(k+1) \leq \wT\log(\wT+1)$. This completes the proof of Theorem~\ref{mmax-thm}.
\end{proof}

\subsection{Proof of the main results in Section \ref{sec:main}} \label{sec4.3}
In this subsection, we provide a proof of our main result presented in Section \ref{sec:main}, which is particularly Theorem \ref{complexity}. Before proceeding, 
let
\begin{align}
\AL_\bfy(x,y,\lambda_\bfy;\rho)=F(x,y)-\frac{1}{2\rho}\left(\|[\lambda_\bfy+\rho d(x,y)]_+\|^2-\|\lambda_\bfy\|^2\right).\label{y-AL}
\end{align}
In view of \eqref{AL}, \eqref{fstarx} and \eqref{y-AL}, one can observe that
\beq\label{p-ineq}
f^*(x)\leq\max_y\AL_\bfy(x,y,\lambda_\bfy;\rho)\qquad \forall x\in\mcX,\ \lambda_\bfy\in\bR_+^{\tm},\ \rho>0,
\eeq
which will be frequently used later. 

We next establish several lemmas that will be used to prove Theorem \ref{complexity} subsequently.  The next lemma provides an upper bound for $\{\lambda^k_\bfy\}_{k\in\bbK}$.

\begin{lemma}\label{l-lycnstr}
Suppose that Assumptions \ref{a1} and \ref{mfcq} hold. Let $\{\lambda^k_\bfy\}_{k\in\bbK}$ be generated by Algorithm \ref{AL-alg},  $D_\bfy$ and $\Delta$ be defined in \eqref{mmax-D} and \eqref{def-r}, and $\tau$, and $\rho_k$ be given in  Algorithm \ref{AL-alg}. Then we have
\beq\label{ly-cnstr}
\rho_k^{-1}\|\lambda^k_\bfy\|^2\leq \|\lambda_\bfy^0\|^2+\frac{2(\Delta+D_\bfy)}{1-\tau}  \qquad \forall 0\leq k\in\bbK-1. 
\eeq
\end{lemma}

\begin{proof}
Its proof is similar to that of \cite[Lemma 5]{lu2024first} and thus omitted.
\end{proof}

The following lemma establishes an upper bound on $\|[d(x^{k+1},y^{k+1})]_+\|$ for $0\leq k\in\bbK-1$.

\begin{lemma}\label{l-ycnstr}
Suppose that Assumptions \ref{a1} and \ref{mfcq} hold. Let $D_\bfy$ and $\Delta$ be defined in \eqref{mmax-D} and \eqref{def-r}, $\delta_d$ be given in Assumption \ref{mfcq}, and $\tau$, $\epsilon_k$ and $\rho_k$ be given in Algorithm \ref{AL-alg}. Suppose that $(x^{k+1},y^{k+1}, \lambda^{k+1}_\bfy)$ is generated by Algorithm \ref{AL-alg} for some $0\leq k\in\bbK-1$ with 
\beq\label{muk-bnd}
\rho_k\geq\frac{4\|\lambda_\bfy^0\|^2}{\delta_d^2}+\frac{8(\Delta+D_\bfy)}{\delta_d^2(1-\tau)}.
\eeq
Then we have
\beq\label{y-cnstr}
\|[d(x^{k+1},y^{k+1})]_+\|\leq\rho_k^{-1}\|\lambda^{k+1}_\bfy\|\leq2 \rho_k^{-1}\delta_d^{-1}(\Delta+D_\bfy).
\eeq
\end{lemma}

\begin{proof}
Its proof is similar to that of \cite[Lemma 6]{lu2024first} and thus omitted.
\end{proof}

\begin{lemma}\label{l-subdcnstr}
Suppose that Assumptions \ref{a1} and \ref{mfcq} hold. Let  $D_\bfy$ and $\Delta$ be defined in  \eqref{mmax-D} and \eqref{def-r}, and $\delta_d$ be given in Assumption \ref{mfcq}, $\tau$, $\epsilon_k$, $\rho_k$ and $\lambda_\bfy^0$ be given in Algorithm \ref{AL-alg}.
Suppose that $(x^{k+1},y^{k+1}, \lambda^{k+1}_\bfx,\lambda^{k+1}_\bfy)$ is generated by Algorithm \ref{AL-alg} for  some $0\leq k\in\bbK-1$ with 
\beq\label{muk-1}
\rho_k\geq\frac{4\|\lambda_\bfy^0\|^2}{\delta_d^2\tau}+\frac{8(\Delta+D_\bfy)}{\delta_d^2\tau(1-\tau)}.
\eeq
Let 
\[
 \tl^{k+1}_\bfx=[\lambda^k_\bfx+\rho_kc(x^{k+1})]_+.
\]
Then we have
\begin{align}
& \dist(0,\partial_x F(x^{k+1},y^{k+1})+\nabla c(x^{k+1})\tl^{k+1}_\bfx -\nabla_xd(x^{k+1},y^{k+1})\lambda^{k+1}_\bfy) \leq \epsilon_k, \nn \\ 
& \dist\left(0,\partial_yF(x^{k+1},y^{k+1})-\nabla_y d(x^{k+1},y^{k+1})\lambda^{k+1}_\bfy\right)\leq\epsilon_k, \nn \\ 
& \|[d(x^{k+1},y^{k+1})]_+\|\leq 2\rho_k^{-1}\delta_d^{-1}(\Delta+D_\bfy), \nn \\ 
&|\langle\lambda^{k+1}_\bfy,d(x^{k+1},y^{k+1})\rangle|\leq 2\rho_k^{-1}\delta_d^{-1}(\Delta+D_\bfy) \max\{\|\lambda_\bfy^0\|,\ 2\delta_d^{-1}(\Delta+D_\bfy)\}. \nn 
\end{align}
\end{lemma}

\begin{proof}
Its proof is similar to that of \cite[Lemma 7]{lu2024first} and thus omitted.
\end{proof}

The following lemma provides an upper bound on $\max_y\AL(x^k_{\rm init},y,\lambda^k_\bfx,\lambda^k_\bfy;\rho_k)$ for $0\leq k\in\bbK-1$, which will subsequently be used to derive an upper bound for  $\max_y\AL(x^{k+1},y,\lambda^k_\bfx,\lambda^k_\bfy;\rho_k)$.

\begin{lemma}
Suppose that Assumptions \ref{a1}, \ref{knownfeas} and \ref{mfcq} hold. Let $\{(\lambda^k_\bfx,\lambda^k_\bfy)\}_{k\in\bbK}$ be generated by Algorithm \ref{AL-alg}, $\AL$, $D_\bfy$, $F_{\rm hi}$ and $\Delta$ be defined in \eqref{AL}, \eqref{mmax-D}, \eqref{Fhi} and \eqref{def-r}, and $\tau$, $\rho_k$, $\Lambda$ and $x^k_{\rm init}$ be given in Algorithm \ref{AL-alg}. Then for all $0\leq k\in\bbK-1$, we have
\beq\label{init-upper}
\max_y\AL(x^k_{\rm init},y,\lambda^k_\bfx,\lambda^k_\bfy;\rho_k)\leq \Delta+F_{\rm hi}+\Lambda+\frac{1}{2}(\tau^{-1}+\|\lambda_\bfy^0\|^2)+\frac{\Delta+D_\bfy}{1-\tau}.
\eeq
\end{lemma}

\begin{proof}
Its proof is similar to that of \cite[Lemma 8]{lu2024first} and thus omitted.
\end{proof}

The next lemma shows that an approximate primal-dual stationary point  of \eqref{AL-sub} is found at each iteration of Algorithm~\ref{AL-alg}, and also provides an estimate of operation complexity for finding it.

\begin{lemma}\label{l-subp}
Suppose that Assumptions \ref{a1}, \ref{knownfeas} and \ref{mfcq} hold. Let $D_\bfx$, $D_\bfy$, $L_k$, $F_{\rm hi}$ and $\Delta$ be defined in  \eqref{mmax-D}, \eqref{Lk}, \eqref{Fhi} and \eqref{def-r}, $\tau$, $\epsilon_k$,  $\rho_k$, $\Lambda$ and $\lambda_\bfy^0$ be given in Algorithm \ref{AL-alg}, and
\begin{align}
\alpha_k=&\ \min\left\{1,\sqrt{8\sigma/L_k}\right\},\label{mmax-omega}\\
\delta_k=&\ (2+\alpha_k^{-1})L_k D_\bfx^2+\max\left\{2\sigma,\alpha_k L_k/4\right\}D_\bfy^2,\label{mmax-zeta}\\
M_k=&\ \frac{16\max\left\{1/(2L_k),\min\left\{1/(2\sigma),4/(\alpha_k L_k)\right\}\right\}\rho_k}{\left[9L_k^2/\min\{L_k,\sigma\}+ 3L_k\right]^{-2}\epsilon_k^2}\nn\\
&\times\Bigg(\delta_k+2\alpha_k^{-1}\bigg(\Delta+\frac{\Lambda^2}{2\rho_k}+\frac{3}{2}\|\lambda_\bfy^0\|^2+\frac{3(\Delta+D_\bfy)}{1-\tau}+\rho_kd_{\rm hi}^2+L_k D_\bfx^2\bigg)\Bigg)\label{mmax-Mk}\\
T_k=&\ \Bigg\lceil16\left(2\Delta+\Lambda+\frac{1}{2}(\tau^{-1}+\|\lambda_\bfy^0\|^2)+\frac{\Delta+D_\bfy}{1-\tau}+\frac{\Lambda^2}{2\rho_k}\right) L_k\epsilon_k^{-2}\nn\\
&\ +8(1+\sigma^{-2}L_k^2)\epsilon_k^2-1\Bigg\rceil_+,\label{mmax-Tk}\\
N_k=&\ 3397\max\left\{2,\sqrt{L_k/(2\sigma)}\right\}\times\left((T_k+1)(\log M_k)_++T_k+1+2T_k\log(T_k+1) \right).\label{mmax-N}
\end{align}
Then for all $0\leq k\in\bbK-1$, Algorithm~\ref{AL-alg} finds an $\epsilon_k$-primal-dual stationary point $(x^{k+1},y^{k+1})$ of problem \eqref{AL-sub} that satisfies 
\begin{align}
\max_y\AL(x^{k+1},y,\lambda^k_\bfx,\lambda^k_\bfy;\rho_k)\leq&\ \Delta+F_{\rm hi}+\Lambda+\frac{1}{2}(\tau^{-1}+\|\lambda_\bfy^0\|^2)+\frac{\Delta+D_\bfy}{1-\tau}\nn\\
&\ +\frac{1}{2}\left(L_k^{-1}+\sigma^{-2}L_k\right)\epsilon_k^2.\label{upperbnd}
\end{align}
Moreover, the total number of evaluations of $\nabla f$, $\nabla c$, $\nabla d$ and proximal operators of $p$ and $q$ performed in iteration $k$ of Algorithm~\ref{AL-alg} is no more than $N_k$, respectively.
\end{lemma}

\begin{proof}
Observe from \eqref{prob} and \eqref{AL} that problem \eqref{AL-sub} can be viewed as
\[
\min_x\max_y\{h(x,y)+p(x)-q(y)\},
\]
where
\[
h(x,y)=f(x,y)+\frac{1}{2\rho _k}\left(\|[\lambda^k_\bfx+\rho_k c(x)]_+\|^2-\|\lambda^k_\bfx\|^2\right)-\frac{1}{2\rho_k}\left(\|[\lambda^k_\bfy+\rho_k d(x,y)]_+\|^2-\|\lambda^k_\bfy\|^2\right).
\]
Notice that
\begin{align*}
& \nabla_x h(x,y)=\nabla_x f(x,y)+\nabla c(x)[\lambda^k_\bfx+\rho_kc(x)]_++\nabla_x d(x,y)[\lambda^k_\bfy+\rho_kd(x,y)]_+, \\
& \nabla_y h(x,y)=\nabla_y f(x,y)+\nabla_y d(x,y)[\lambda^k_\bfy+\rho_kd(x,y)]_+.
\end{align*}
It follows from Assumption \ref{a1}(iii) that 
\[
\|\nabla c(x)\|\leq L_c,\quad\|\nabla d(x,y)\| \leq L_d \qquad \forall (x,y)\in\mcX\times\mcY.
\]
In view of the above relations, \eqref{cdhi} and Assumption \ref{a1}, one can observe that  $\nabla c(x)[\lambda^k_\bfx+\rho_kc(x)]_+$ is $(\rho_kL_c^2+\rho_kc_{\rm hi}L_{\nabla c}+\|\lambda^k_\bfx\| L_{\nabla c})$-Lipschitz continuous on $\mcX$, and $\nabla d(x,y)[\lambda^k_\bfy+\rho_kd(x,y)]_+$ is $(\rho_kL_d^2+\rho_kd_{\rm hi}L_{\nabla d}+\|\lambda^k_\bfy\|L_{\nabla d})$-Lipschitz continuous on $\mcX\times\mcY$. Using these and the fact that $\nabla f(x,y)$ is $L_{\nabla f}$-Lipschitz continuous on $\mcX\times\mcY$ and $f(x,\cdot)$ is $\sigma$-strongly-concave on $\mcY$ for all $x\in\mcX$, we can see that $h(x,\cdot)$ is $\sigma$-strongly-concave on $\mcY$, and $h(x,y)$ is $L_k$-smooth on $\mcX\times\mcY$ for all $0\leq k\in\bbK-1$, where $L_k$ is given in \eqref{Lk}. Consequently, it follows from Theorem \ref{mmax-thm} that Algorithm \ref{mmax-alg2} can be suitably applied to problem \eqref{AL-sub} for finding an $\epsilon_k$-primal-dual stationary point $(x^{k+1},y^{k+1})$ of it.

In addition, by \eqref{AL}, \eqref{F-gap}, \eqref{y-AL}, \eqref{p-ineq} and $\|\lambda^k_\bfx\|\leq\Lambda$ (see Algorithm \ref{AL-alg}), one has
\begin{align}
&\ \min_x\max_y\AL(x,y,\lambda^k_\bfx,\lambda^k_\bfy;\rho_k)\overset{\eqref{AL} \eqref{y-AL}}=\min_x\max_y\left\{\AL_\bfy(x,y,\lambda^k_\bfy;\rho_k)+\frac{1}{2\rho_k}\left(\|[\lambda^k_\bfx+\rho_k c(x)]_+\|^2-\|\lambda^k_\bfx\|^2\right)\right\}\nn\\
&\overset{\eqref{p-ineq}}{\geq}\min_x \left\{f^*(x)+\frac{1}{2\rho_k}\left(\|[\lambda^k_\bfx+\rho_k c(x)]_+\|^2-\|\lambda^k_\bfx\|^2\right)\right\} \overset{\eqref{F-gap}}{\geq} F_{\rm low}-\frac{1}{2\rho_k}\|\lambda^k_\bfx\|^2\geq F_{\rm low}-\frac{\Lambda^2}{2\rho_k}.\label{e1}
\end{align}
Let $(x^*,y^*)$ be an optimal solution of \eqref{prob}. It then follows that $c(x^*)\leq0$. Using this, \eqref{AL}, \eqref{Fhi} and \eqref{ly-cnstr}, we obtain that
\begin{align}
&\ \ \min_x\max_y\AL(x,y,\lambda^k_\bfx,\lambda^k_\bfy;\rho_k)\leq\max_y\AL(x^*,y,\lambda^k_\bfx,\lambda^k_\bfy;\rho_k)\nn\\
&\ \overset{\eqref{AL}}{=}\max_y\left\{F(x^*,y)+\frac{1}{2\rho_k}\left(\|[\lambda^k_\bfx+\rho_kc(x^*)]_+\|^2-\|\lambda^k_\bfx\|^2\right)-\frac{1}{2\rho_k}\left(\|[\lambda^k_\bfy+\rho_k d(x^*,y)]_+\|^2-\|\lambda^k_\bfy\|^2\right)\right\}\nn\\
&\ \leq\ \max_y\left\{F(x^*,y)-\frac{1}{2\rho_k}\left(\|[\lambda^k_\bfy+\rho_k d(x^*,y)]_+\|^2-\|\lambda^k_\bfy\|^2\right)\right\}\nn\\
&\  \overset{\eqref{Fhi}}{\leq} F_{\rm hi}+\frac{1}{2\rho_k}\|\lambda^k_\bfy\|^2\overset{ \eqref{ly-cnstr}}{\leq}F_{\rm hi}+\frac{1}{2}\|\lambda_\bfy^0\|^2+\frac{\Delta+D_\bfy}{1-\tau},\label{e2}
\end{align}
where the second inequality is due to $c(x^*)\leq0$. Moreover, it follows from this, \eqref{AL}, \eqref{cdhi}, \eqref{Fhi}, \eqref{ly-cnstr}, $\lambda^k_\bfy\in\bR_+^{\tm}$ and $\|\lambda^k_\bfx\|\leq\Lambda$ that
\begin{align}
& \min_{(x,y)\in\mcX\times\mcY}\AL(x,y,\lambda^k_\bfx,\lambda^k_\bfy;\rho_k)\overset{\eqref{AL}}{\geq}\min_{(x,y)\in\mcX\times\mcY}\left\{F(x,y)-\frac{1}{2\rho_k}\|\lambda^k_\bfx\|^2-\frac{1}{2\rho_k}\|[\lambda^k_\bfy+\rho_kd(x,y)]_+\|^2\right\}\nn\\
&\ \geq\min_{(x,y)\in\mcX\times\mcY}\left\{F(x,y)-\frac{1}{2\rho_k}\|\lambda^k_\bfx\|^2-\frac{1}{2\rho_k}\left(\|\lambda^k_\bfy\|+\rho_k\|[d(x,y)]_+\|\right)^2\right\}\nn\\
&\ \geq\min_{(x,y)\in\mcX\times\mcY}\left\{F(x,y)-\frac{1}{2\rho_k}\|\lambda^k_\bfx\|^2-\rho_k^{-1}\|\lambda^k_\bfy\|^2-\rho_k\|[d(x,y)]_+\|^2\right\}\nn\\
&\ \geq F_{\rm low}-\frac{\Lambda^2}{2\rho_k}-\|\lambda_\bfy^0\|^2-\frac{2(\Delta+D_\bfy)}{1-\tau}-\rho_kd_{\rm hi}^2,\label{e3}
\end{align}
where the second inequality is due to $\lambda^k_\bfy\in\bR_+^{\tm}$ and the last inequality is due to \eqref{cdhi}, \eqref{Fhi}, \eqref{ly-cnstr} and $\|\lambda^k_\bfx\|\leq\Lambda$.

To complete the rest of the proof, let
\begin{align}
&H(x,y)=\AL(x,y,\lambda^k_\bfx,\lambda^k_\bfy;\rho_k),\quad H^*=\min_x\max_y\AL(x,y,\lambda^k_\bfx,\lambda^k_\bfy;\rho_k),\label{Hk1}\\
&H_{\rm low}=\min_{(x,y)\in\mcX\times\mcY}\AL(x,y,\lambda^k_\bfx,\lambda^k_\bfy;\rho_k).\label{Hk2}
\end{align}
In view of these, \eqref{init-upper}, \eqref{e1}, \eqref{e2}, \eqref{e3}, we obtain that
\begin{align*}
&\max_yH(x^k_{\rm init},y)\overset{\eqref{init-upper}}{\leq}\Delta+F_{\rm hi}+\Lambda+\frac{1}{2}(\tau^{-1}+\|\lambda_\bfy^0\|^2)+\frac{\Delta+D_\bfy}{1-\tau},\\
&F_{\rm low}-\frac{\Lambda^2}{2\rho_k}\overset{\eqref{e1}}{\leq} H^*\overset{\eqref{e2}}{\leq} F_{\rm hi}+\frac{1}{2}\|\lambda_\bfy^0\|^2+\frac{\Delta+D_\bfy}{1-\tau},\\
&H_{\rm low}\overset{\eqref{e3}}{\geq}F_{\rm low}-\frac{\Lambda^2}{2\rho_k}-\|\lambda_\bfy^0\|^2-\frac{2(\Delta+D_\bfy)}{1-\tau}-\rho_kd_{\rm hi}^2.
\end{align*}
Using these, \eqref{def-r}, and Theorem \ref{mmax-thm} with $x^0=x^k_{\rm init}$, $\epsilon=\epsilon_k$, $\hat\epsilon_0=\epsilon_k/2$, $L_{\nabla\bh}=L_k$, $\sigma_y=\sigma$,  $\halpha=\alpha_k$, $\hdelta=\delta_k$, and $\bH$, $\bH^*$, $\bH_{\rm low}$ given in \eqref{Hk1} and \eqref{Hk2}, we can conclude that Algorithm \ref{mmax-alg2} performs at most $N_k$ evaluations of $\nabla f$, $\nabla c$, $\nabla d$ and proximal operators of $p$ and $q$ for finding an $\epsilon_k$-primal-dual stationary point of problem \eqref{AL-sub} satisfying \eqref{upperbnd}.
\end{proof}

The following lemma provides an upper bound on $\|[c(x^{k+1})]_+\|$ and $ |\langle\tl^{k+1}_\bfx,c(x^{k+1})\rangle| $ for $0\leq k\in\bbK-1$, where $\tl^{k+1}_\bfx$ is given below.

\begin{lemma}\label{l-xcnstr2}
Suppose that Assumptions \ref{a1}, \ref{knownfeas} and \ref{mfcq} hold. Let $D_\bfy$, $\Delta$ and $L$ be defined in \eqref{mmax-D}, \eqref{def-r}  and \eqref{hL}, $L_F$, $L_c$, $\delta_c$ and $\theta$ be given in Assumption \ref{mfcq}, and $\tau$, $\rho_k$, $\Lambda$ and $\lambda_\bfy^0$ be given in Algorithm \ref{AL-alg}. Suppose that $(x^{k+1},\lambda^{k+1}_\bfx)$ is generated by Algorithm \ref{AL-alg} for some $0\leq k\in\bbK-1$ with 
\begin{align}
\rho_k \geq \max\Bigg\{&\theta^{-1}\Lambda, \theta^{-2}\Big\{4\Delta+2\Lambda+\tau^{-1}+\|\lambda_\bfy^0\|^2+\frac{2(\Delta+D_\bfy)}{1-\tau} \nn \\
& +L_c^{-2} +\sigma^{-2}L+\Lambda^2\Big\}, \frac{4\|\lambda_\bfy^0\|^2}{\delta_d^2\tau}+\frac{8(\Delta+D_\bfy)}{\delta_d^2\tau(1-\tau)}\Bigg\}. \label{rhok-bnd}
\end{align}
Let 
\beq \label{def-tlx1} 
 \tl^{k+1}_\bfx=[\lambda^k_\bfx+\rho_kc(x^{k+1})]_+.
\eeq
Then we have
\begin{align}
&\|[c(x^{k+1})]_+\|\leq \rho_k^{-1}\delta_c^{-1}\left(L_F +2L_d\delta_d^{-1}(\Delta+D_\bfy)+1\right), \label{x-cnstr2} \\
& |\langle\tl^{k+1}_\bfx,c(x^{k+1})\rangle| \leq \rho^{-1}_k \delta_c^{-1}(L_F +2L_d\delta_d^{-1}(\Delta+D_\bfy)+1)\max\{\delta_c^{-1}(L_F +2L_d\delta_d^{-1}(\Delta+D_\bfy)+1), \Lambda\}. \label{c-complim}
\end{align}
\end{lemma}

\begin{proof}
One can observe from \eqref{AL}, \eqref{F-gap}, \eqref{y-AL} and \eqref{p-ineq} that
\begin{align*}
\max_y\AL(x^{k+1},y,\lambda^k_\bfx,\lambda^k_\bfy;\rho_k)=\ &\ \max_y\AL_\bfy(x^{k+1},y,\lambda^k_\bfy;\rho_k)+\frac{1}{2\rho_k}\left(\|[\lambda^k_\bfx+\rho_k c(x^{k+1})]_+\|^2-\|\lambda^k_\bfx\|^2\right)\\
\overset{\eqref{p-ineq}}{\geq}&\ f^*(x^{k+1})+\frac{1}{2\rho_k}\left(\|[\lambda^k_\bfx+\rho_k c(x^{k+1})]_+\|^2-\|\lambda^k_\bfx\|^2\right)\\
\overset{\eqref{F-gap}}{\geq}\ &\ F_{\rm low}+\frac{1}{2\rho_k}\left(\|[\lambda^k_\bfx+\rho_k c(x^{k+1})]_+\|^2-\|\lambda^k_\bfx\|^2\right).
\end{align*}
By this inequality, \eqref{upperbnd} and $\|\lambda^k_\bfx\|\leq\Lambda$, one has
\begin{align*}
&\|[\lambda^k_\bfx+\rho_kc(x^{k+1})]_+\|^2\leq 2\rho_k\max_y\AL(x^{k+1},y,\lambda^k_\bfx,\lambda^k_\bfy;\rho_k)-2\rho_kF_{\rm low}+\|\lambda^k_\bfx\|^2\\
&\leq2\rho_k\max_y\AL(x^{k+1},y,\lambda^k_\bfx,\lambda^k_\bfy;\rho_k)-2\rho_k F_{\rm low}+\Lambda^2\\
&\overset{\eqref{upperbnd}}\leq 2\rho_k \Delta+2\rho_kF_{\rm hi}+2\rho_k\Lambda+\rho_k(\tau^{-1}+\|\lambda_\bfy^0\|^2)+\frac{2\rho_k(Delta+D_\bfy)}{1-\tau}\\
&\qquad +L_k^{-1}\epsilon_k^2+\sigma^{-2}L_k\epsilon_k^2-2\rho_k F_{\rm low}+\Lambda^2.
\end{align*}
This together with \eqref{def-r} and $\rho_k^2\|[c(x^{k+1})]_+\|^2\leq\|[\lambda^k_\bfx+\rho_kc(x^{k+1})]_+\|^2$ implies that
\begin{align} 
\|[c(x^{k+1})]_+\|^2\leq&\ \rho_k^{-1}\left(4\Delta+2\Lambda+\tau^{-1}+\|\lambda_\bfy^0\|^2+\frac{2(\Delta+D_\bfy)}{1-\tau}\right)\nn\\
&\ +\rho_k^{-2}\left(L_k^{-1}\epsilon_k^2+\sigma^{-2}L_k\epsilon_k^2+\Lambda^2\right). \label{x-cnstr}
\end{align}
In addition, we observe from \eqref{Lk}, \eqref{hL}, \eqref{ly-cnstr}, $\rho_k\geq1$ and $\|\lambda^k_\bfx\|\leq\Lambda$ that for all $0\leq k\leq K$,
\begin{align}
& \rho_kL_c^2\leq L_k= L_{\nabla f}+\rho_kL_c^2+\rho_kc_{\rm hi}L_{\nabla c}+\|\lambda^k_\bfx\|L_{\nabla c}+\rho_kL_d^2+\rho_kd_{\rm hi}L_{\nabla d}+\|\lambda^k_\bfy\|L_{\nabla d}\nn\\
&\leq L_{\nabla f}+\rho_kL_c^2+\rho_kc_{\rm hi}L_{\nabla c}+\Lambda L_{\nabla c}+\rho_kL_d^2+\rho_kd_{\rm hi}L_{\nabla d}\nn\\
&\ \ \ \ +L_{\nabla d}\sqrt{\rho_k\left(\|\lambda_\bfy^0\|^2+\frac{2(\Delta+D_\bfy)}{1-\tau}\right)}\leq\rho_kL.\label{L-ineq}
\end{align}
 Using this relation, \eqref{rhok-bnd}, \eqref{x-cnstr},  $\rho_k\geq1$ and $\epsilon_k \leq 1$, we have
\begin{align*}
\|[c(x^{k+1})]_+\|^2\leq&\ \rho_k^{-1}\left(4\Delta+2\Lambda+\tau^{-1}+\|\lambda_\bfy^0\|^2+\frac{2(\Delta+D_\bfy)}{1-\tau}\right)\nn\\
& \ +\rho_k^{-2}\left((\rho_kL_c^2)^{-1}\epsilon_k^2+\sigma^{-2}L\epsilon_k^2\rho_k+\Lambda^2\right)\\
\leq\ &\ \rho_k^{-1}\left(4\Delta+2\Lambda+\tau^{-1}+\|\lambda_\bfy^0\|^2+\frac{2(\Delta+D_\bfy)}{1-\tau}\right)\nn\\
& \ +\rho_k^{-1}\left(L_c^{-2}+4\sigma^{-2}L+\Lambda^2\right)\overset{\eqref{rhok-bnd}}{\leq}\theta^2,
\end{align*}
which together with \eqref{def-cAcS} implies that $x^{k+1}\in\cF(\theta)$.  

It follows from $x^{k+1}\in\cF(\theta)$ and Assumption \ref{mfcq}(i) that there exists some $v\in\mcT_{\mcX}(x^{k+1})$ such that $\|v\|=1$ and $v^T\nabla c_i(x^{k+1})\leq-\delta_c$ for all $i\in\cA(x^{k+1};\theta)$, where $\cA(x^{k+1};\theta)$ is defined in \eqref{def-cAcS}. Let $\bar\cA(x^{k+1};\theta)=\{1,2,\ldots,\tn\}\backslash\cA(x^{k+1};\theta)$. Notice from \eqref{def-cAcS} that $c_i(x^{k+1})<-\theta$ for all $i\in\bar\cA(x^{k+1};\theta)$.
In addition, observe from \eqref{rhok-bnd} that $\rho_k\geq\theta^{-1}\Lambda$. 
Using these and $\|\lambda^k_\bfx\|\leq\Lambda$, we obtain that $(\lambda^k_\bfx+\rho_kc(x^{k+1}))_i\leq\Lambda-\rho_k\theta\leq0$ for all $i\in\bar\cA(x^{k+1};\theta)$. By this and the fact that $v^T\nabla c_i(x^{k+1})\leq-\delta_c$ for all $i\in\cA(x^{k+1};\theta)$, one has
\begin{align}
&v^T\nabla c(x^{k+1})\tl^{k+1}_\bfx\overset{\eqref{def-tlx1}}=v^T\nabla c(x^{k+1})[\lambda^k_\bfx+\rho_kc(x^{k+1})]_+=\sum_{i=1}^\tn v^T\nabla c_i(x^{k+1})([\lambda^k_\bfx+\rho_kc(x^{k+1})]_+)_i\nn\\
&=\sum_{i\in\cA(x^{k+1};\theta)}v^T\nabla c_i(x^{k+1})([\lambda^k_\bfx+\rho_kc(x^{k+1})]_+)_i+\sum_{i\in\bar\cA(x^{k+1};\theta)}v^T\nabla c_i(x^{k+1})([\lambda^k_\bfx+\rho_kc(x^{k+1})]_+)_i\nn\\
&\leq-\delta_c \sum_{i\in\cA(x^{k+1};\theta)}([\lambda^k_\bfx+\rho_kc(x^{k+1})]_+)_i=-\delta_c  
\sum_{i=1}^\tn([\lambda^k_\bfx+\rho_kc(x^{k+1})]_+)_i \overset{\eqref{def-tlx1}}= -\delta_c\|\tl^{k+1}_\bfx\|_1. \label{mfcq-ineq}
\end{align}

Since $(x^{k+1},y^{k+1})$ is an $\epsilon_k$-primal-dual stationary point of \eqref{AL-sub}, it follows from \eqref{AL} and Definition \ref{def2} that there exists some $s\in\partial_xF(x^{k+1},y^{k+1})$ such that
\[
\|s+\nabla c(x^{k+1})[\lambda^k_\bfx+\rho_kc(x^{k+1})]_+-\nabla_xd(x^{k+1},y^{k+1})[\lambda^k_\bfy+\rho_kd(x^{k+1},y^{k+1})]_+\|\leq\epsilon_k,
\]
which along with \eqref{def-tlx1} and $\lambda^{k+1}_\bfy=[\lambda^k_\bfy+\rho_xd(x^{k+1},y^{k+1})]_+$ implies that
\beq\label{eps-subgrad}
\|s+\nabla c(x^{k+1})\tl^{k+1}_\bfx -\nabla_xd(x^{k+1},y^{k+1})\lambda^{k+1}_\bfy\|\leq\epsilon_k.
\eeq
In addition, since $v\in\mcT_{\mcX}(x^{k+1})$, there exist $\{z^t\} \subset \mcX$ and $\{\alpha_t\} \downarrow 0$ such that $z^t=x^{k+1}+\alpha_t v+o(\alpha_t)$ for all $t$.  Also, since $s\in\partial_xF(x^{k+1},y^{k+1})$, one has $s=\nabla_x f(x^{k+1},y^{k+1})+s_p$ for some $s_p\in\partial p(x^{k+1})$. Using these and Assumptions \ref{a1} and  \ref{mfcq}(iii), we have
\begin{align}
\langle s, v \rangle &=\langle \nabla_x f(x^{k+1},y^{k+1}), v \rangle + \lim_{t \to \infty} \alpha_t^{-1}\langle s_p, z^t-x^{k+1}\rangle  \nn \\
& = \lim_{t \to \infty} \alpha_t^{-1} (f(z^t,y^{k+1})-f(x^{k+1},y^{k+1})) + \lim_{t \to \infty} \alpha_t^{-1}\langle s_p, z^t-x^{k+1}\rangle \nn \\
& \leq \lim_{t \to \infty} \alpha_t^{-1} (f(z^t,y^{k+1})-f(x^{k+1},y^{k+1})) + \lim_{t \to \infty} \alpha_t^{-1} (p(z^t)-p(x^{k+1}))  \nn \\
& = \lim_{t \to \infty} \alpha_t^{-1} (F(z^t,y^{k+1})-F(x^{k+1},y^{k+1})) \leq L_F \lim_{t \to \infty} \alpha_t^{-1} \|z^t-x^{k+1}\| =L_F, \label{s-bnd}
\end{align}
where the second equality is due to the differentiability of $f$, the first inequality follows from the convexity of $p$ and $s_p\in\partial p(x^{k+1})$, the second inequality is due to the $L_F$-Lipschitz continuity of $F(\cdot, y^{k+1})$, and the last equality follows from $\lim_{t \to \infty} \alpha_t^{-1} \|z^t-x^{k+1}\|=\|v\|=1$.

By \eqref{mfcq-ineq}, \eqref{eps-subgrad}, \eqref{s-bnd},  and $\|v\|=1$, one has
\begin{align*}
\epsilon_k& \geq\|s+\nabla c(x^{k+1})\tl^{k+1}_\bfx -\nabla_xd(x^{k+1},y^{k+1})\lambda^{k+1}_\bfy\|\cdot\|v\|\\
&\geq\langle s+\nabla c(x^{k+1})\tl^{k+1}_\bfx -\nabla_xd(x^{k+1},y^{k+1})\lambda^{k+1}_\bfy,-v\rangle\\
&=-\langle s-\nabla_xd(x^{k+1},y^{k+1})\lambda^{k+1}_\bfy,v\rangle-v^T\nabla c(x^{k+1})\tl^{k+1}_\bfx\\
&\overset{\eqref{mfcq-ineq}}\geq -\langle s,v\rangle-\|\nabla_xd(x^{k+1},y^{k+1})\|\|\lambda^{k+1}_\bfy\|\|v\|+\delta_c\|\tl^{k+1}_\bfx\|_1 \\
&\geq -L_F -L_d\|\lambda^{k+1}_\bfy\|+\delta_c\|\tl^{k+1}_\bfx\|_1,
\end{align*}
where the last inequality is due to $\|v\|=1$ and Assumptions \ref{a1}(i) and \ref{a1}(iii).
Notice from \eqref{rhok-bnd} that \eqref{muk-bnd} holds. It then follows from \eqref{y-cnstr} that $\|\lambda^{k+1}_\bfy\|\leq2\delta_d^{-1}(\Delta+D_\bfy)$, which together with the above inequality and $\epsilon_k\leq 1$ yields
\beq \label{tl-bnd}
\|\tl^{k+1}_\bfx\| \leq \|\tl^{k+1}_\bfx\|_1 \leq \delta_c^{-1}(L_F +L_d\|\lambda^{k+1}_\bfy\|+\epsilon_k)\leq \delta_c^{-1}(L_F +2L_d\delta_d^{-1}(\Delta+D_\bfy)+1).
\eeq
By this and \eqref{def-tlx1}, one can observe that
\[
\|[c(x^{k+1})]_+\| \le \rho^{-1}_k \|[\lambda^k_\bfx+\rho_kc(x^{k+1})]_+\| = \rho^{-1}_k \|\tl^{k+1}_\bfx\| \leq \rho^{-1}_k\delta_c^{-1}(L_F +2L_d\delta_d^{-1}(\Delta+D_\bfy)+1).
\]
Hence, \eqref{x-cnstr2} holds as desired. 

We next show that \eqref{c-complim} holds. Indeed, by $\tl^{k+1}_\bfx\geq 0$, \eqref{x-cnstr2} and \eqref{tl-bnd}, one has 
\begin{align}
\langle\tl^{k+1}_\bfx,c(x^{k+1})\rangle & \ \ \leq \ \langle\tl^{k+1}_\bfx,[c(x^{k+1})]_+\rangle \leq\|\tl^{k+1}_\bfx\|\|[c(x^{k+1})]_+\|\nn \\ 
& \overset{\eqref{x-cnstr2}\eqref{tl-bnd}}\leq\rho^{-1}_k\delta_c^{-2}(L_F +2L_d\delta_d^{-1}(\Delta+D_\bfy)+1)^2. \label{complim-ineq2}
\end{align}
Notice that $\langle\lambda^{k+1}_\bfx,\lambda^k_\bfx+\rho_kc(x^{k+1})\rangle=\|[\lambda^k_\bfx+\rho_kc(x^{k+1})]_+\|^2\geq0$. Hence, we have
\[
-\langle\tl^{k+1}_\bfx,\rho_k^{-1}\lambda^k_\bfx\rangle
\leq \langle\tl^{k+1}_\bfx,c(x^{k+1})\rangle,
\]
which along with $\|\lambda^k_\bfx\|\leq\Lambda$ and \eqref{tl-bnd} yields
\[
\langle\tl^{k+1}_\bfx,c(x^{k+1})\rangle\geq-\rho_k^{-1}\|\tl^{k+1}_\bfx\|\|\lambda^k_\bfx\| \geq -\rho^{-1}_k \delta_c^{-1}(L_F +2L_d\delta_d^{-1}(\Delta+D_\bfy)+1)\Lambda.
\]
The relation \eqref{c-complim} then follows from this and \eqref{complim-ineq2}.
\end{proof}

We are now ready to prove Theorem \ref{complexity} using Lemmas \ref{l-subdcnstr}, \ref{l-subp} and \ref{l-xcnstr2}.

\begin{proof}[\textbf{Proof of Theorem \ref{complexity}}]
(i) Observe from the definition of $K$ in \eqref{K1} and $\epsilon_k=\tau^k$ that $K$ is the smallest nonnegative integer such that $\epsilon_K\leq\varepsilon$. Hence, Algorithm \ref{AL-alg} terminates and outputs $(x^{K+1},y^{K+1})$ after $K+1$ outer iterations. It follows from these and $\rho_k=\epsilon_k^{-1}$ that $\epsilon_K\leq\varepsilon$ and $\rho_K\geq\varepsilon^{-1}$. By this and  \eqref{cond}, one can see that \eqref{muk-1} and \eqref{rhok-bnd} holds for $k=K$. It then follows from Lemmas \ref{l-subdcnstr} and \ref{l-xcnstr2} that \eqref{t1-1}-\eqref{t1-6} hold. 

(ii) Let $K$ and $N$ be given in \eqref{K1} and \eqref{N2}. Recall from Lemma \ref{l-subp} that the number of evaluations of $\nabla f$, $\nabla c$, $\nabla d$, proximal operators of $p$ and $q$ performed by Algorithm \ref{mmax-alg2} at iteration $k$ of Algorithm~\ref{AL-alg} is at most $N_k$, where $N_k$ is given in \eqref{mmax-N}. By this and statement (i) of this theorem, one can observe that the total number of evaluations of $\nabla f$, $\nabla c$, $\nabla d$, proximal operators of $p$ and $q$ performed in Algorithm~\ref{AL-alg} is no more than $\sum_{k=0}^KN_k$, respectively. As a result, to prove statement (ii) of this theorem, it suffices to show that $\sum_{k=0}^KN_k\leq N$. Recall from  \eqref{L-ineq} and Algorithm \ref{AL-alg} that  $\rho_kL_c^2\leq L_k\leq\rho_kL$ and $\rho_k\geq1\geq\epsilon_k$. Using these, \eqref{ho}, \eqref{hM}, \eqref{hT}, \eqref{mmax-omega}, \eqref{mmax-zeta}, \eqref{mmax-Mk} and \eqref{mmax-Tk}, we obtain that
\begin{align}
&1\geq\alpha_k\geq\min\left\{1,\sqrt{8\sigma/(\rho_kL)}\right\}\geq\rho_k^{-1/2}\alpha, \label{alpha-ineq}\\
&\delta_k\leq(2+\rho_k^{1/2}\alpha^{-1})\rho_kL D_\bfx^2+\max\{2\sigma,\rho_kL/4\}D_\bfy^2\leq\rho_k^{3/2}\delta, \label{delta-ineq}\\
&M_k\leq\frac{16\max\left\{1/(2\rho_kL_c^2),4/(\rho_k^{-1/2}\alpha\rho_kL_c^2)\right\}}{\left[9\rho_k^2L^2/\min\{\rho_kL_c^2,\sigma\}+3\rho_kL\right]^{-2}\epsilon_k^2}\times\Bigg(\rho_k^{3/2}\delta+2\rho_k^{1/2}\alpha^{-1}\nn\\
&\ \ \ \ \ \ \ \times\Big(\Delta+\frac{\Lambda^2}{2}+\frac{3}{2}\|\lambda_\bfy^0\|^2+\frac{3(\Delta+D_\bfy)}{1-\tau}+\rho_kd_{\rm hi}^2+\rho_kL D_\bfx^2\Big)\Bigg) \label{M-ineq} \\
&\leq \frac{16\rho_k^{-1/2}\max\left\{1/(2L_c^2),4/(\alpha L_c^2)\right\}}{\rho_k^{-4}\left[9L^2/\min\{L_c^2,\sigma\}+3L\right]^{-2}\epsilon_k^2}\times\rho_k^{3/2}\Bigg(\delta+2\alpha^{-1}\nn \\
&\ \ \times\Big(\Delta+\frac{\Lambda^2}{2}+\frac{3}{2}\|\lambda_\bfy^0\|^2+\frac{3(\Delta+D_\bfy)}{1-\tau}+d_{\rm hi}^2+L D_\bfx^2\Big)\Bigg)\leq\epsilon_k^{-2}\rho_k^5M, \nn \\
&T_k\leq\Bigg\lceil16\left(2\Delta+\Lambda+\frac{1}{2}(\tau^{-1}+\|\lambda_\bfy^0\|^2)+\frac{\Delta+D_\bfy}{1-\tau}+\frac{\Lambda^2}{2}\right)\epsilon_k^{-2}\rho_kL\nn \\
&\ \ \ \ \ \ \ \ +8(1+\sigma^{-2}\rho_k^2L^2)\epsilon_k^{-2}-1\Bigg\rceil_+\leq\epsilon_k^{-2}\rho_kT, \nn
\end{align}
where \eqref{M-ineq} follows from \eqref{ho}, \eqref{hM}, \eqref{hT}, \eqref{alpha-ineq}, \eqref{delta-ineq}, $\rho_kL_c^2\leq L_k\leq\rho_kL$, and $\rho_k\geq1\geq\epsilon_k$. 
By the above inequalities, \eqref{mmax-N}, \eqref{L-ineq}, $T\geq1$ and $\rho_k\geq1\geq\epsilon_k$, one has
\begin{align}
&\sum_{k=0}^KN_k\leq \sum_{k=0}^K3397\max\left\{2,\sqrt{\rho_kL/(2\sigma)}\right\}\nn\\
&\ \ \ \ \ \ \ \ \ \ \ \ \ \times\left((\epsilon_k^{-2}\rho_kT+1)(\log (\epsilon_k^{-2}\rho_k^5M))_++\epsilon_k^{-2}\rho_kT+1+2\epsilon_k^{-2}\rho_kT\log(\epsilon_k^{-2}\rho_kT+1) \right)\nn\\
&\leq\sum_{k=0}^K3397\max\left\{2,\sqrt{L/(2\sigma)}\right\} \times\epsilon_k^{-2}\rho_k^{3/2}\left((T+1)(\log (\epsilon_k^{-2}\rho_k^5M))_++T+1+2T\log(\epsilon_k^{-2}\rho_kT+1) \right)\nn\\
&\leq\sum_{k=0}^K3397\max\left\{2,\sqrt{L/(2\sigma)}\right\} T\epsilon_k^{-2}\rho_k^{3/2}\left((2\log (\epsilon_k^{-2}\rho_k^5M))_++2+2\log(2\epsilon_k^{-2}\rho_kT) \right)\nn\\
&\leq\sum_{k=0}^K3397\max\left\{2,\sqrt{L/(2\sigma)}\right\}T\epsilon_k^{-2}\rho_k^{3/2}\left(12\log\rho_k-8\log\epsilon_k+2(\log M)_++2+2\log(2T) \right),\label{sum-N}
\end{align}
By the definition of $K$ in \eqref{K1}, one has $\tau^K\geq\tau\varepsilon$. Also, notice from Algorithm \ref{AL-alg} that $\rho_k=\tau^{-k}$. It then follows from these, \eqref{N2} and \eqref{sum-N} that
\begin{align*}
&\sum_{k=0}^KN_k\leq\sum_{k=0}^K3397\max\left\{2,\sqrt{L/(2\sigma)}\right\}T\epsilon_k^{-7/2}\left(20\log(1/\epsilon_k)+2(\log M)_++2+2\log(2T) \right)\\
&=3397\max\left\{2,\sqrt{L/(2\sigma)}\right\}T\sum_{k=0}^K\tau^{-7k/2}\left(20k\log(1/\tau)+2(\log M)_++2+2\log(2T) \right)\\
&\leq 3397\max\left\{2,\sqrt{L/(2\sigma)}\right\}T\sum_{k=0}^K\tau^{-7k/2}\left(20K\log(1/\tau)+2(\log M)_++2+2\log(2T) \right)\\
&\leq 3397\max\left\{2,\sqrt{L/(2\sigma)}\right\}T\tau^{-7/2K}(1-\tau^4)^{-1}\\
&\ \ \ \times\left(20K\log(1/\tau)+2(\log M)_++2+2\log(2T) \right)\\
&\leq 3397\max\left\{2,\sqrt{L/(2\sigma)}\right\}T(1-\tau^{7/2})^{-1}\\
&\ \ \ \times(\tau\varepsilon)^{-7/2}\left(20K\log(1/\tau)+2(\log M)_++2+2\log(2T) \right)\overset{\eqref{N2}}{=} N,
\end{align*}
where the second last inequality is due to $\sum_{k=0}^K\tau^{-7k/2}\leq \tau^{-7K/2}/(1-\tau^{7/2})$, and the last inequality is due to $\tau^K\geq\tau\varepsilon$. Hence, statement (ii) of this theorem holds as desired.
\end{proof}

\section*{Declarations}

The authors report there are no competing interests to declare.

\appendix

\section{A modified optimal first-order method for strongly-convex-strongly-concave minimax problem}
\label{strong-cvx-ccv}

In this part, we present a modified optimal first-order method \cite[Algorithm 1]{lu2024first} in Algorithm \ref{mmax-alg1} below for finding an approximate primal-dual stationary point of strongly-convex-strongly-concave minimax problem
\begin{equation}\label{ea-prob}
\H^*=\min_{x}\max_{y}\left\{\H(x,y)\coloneqq \h(x,y)+\np(x)-\nq(y)\right\},
\end{equation}
which satisfies the following assumptions.
\begin{assumption}\label{ea}
\begin{enumerate}[label=(\roman*)]
\item $p:\bR^n\to\bR\cup\{\infty\}$ and $q:\bR^m\to\bR\cup\{\infty\}$ are proper convex functions and continuous on $\dom\,p$ and $\dom\,q$, respectively, and moreover, $\dom\,p$ and $\dom\,q$ are compact.
\item The proximal operators associated with $p$ and $q$ can be exactly evaluated.
\item $\h(x,y)$ is $\sigma_x$-strongly-convex-$\sigma_y$-strongly-concave and $L_{\nabla\h}$-smooth on $\dom\,\np\times\dom\,\nq$ for some $\sigma_x,\sigma_y>0$.
\end{enumerate}
\end{assumption}

For convenience of presentation, we introduce some notation below, most of which is adopted from \cite{kovalev2022first}. Let $\mcX=\dom\,p$, $\mcY=\dom\,q$,   $(x^*,y^*)$ denote the optimal solution of \eqref{ea-prob}, $z^*=-\sigma_x x^*$,  and
\begin{align}
&D_\bfx\coloneqq \max\{\|u-v\|\big|u,v\in\mcX\},\quad D_\bfy\coloneqq\max\{\|u-v\|\big|u,v\in\mcY\}, \label{mmax-D1}\\
&\H_{\rm low}=\min\left\{\H(x,y)|\right(x,y)\in\mcX \times\mcY\}, \label{ea-bnd} \\
&\hat h(x,y)=\h(x,y)-\sigma_x\|x\|^2/2+\sigma_y\|y\|^2/2,\nn \\
&\cG(z,y)=\sup_{x}\{\langle x,z\rangle-\np(x)-\hat h(x,y)+\nq(y)\},\nn\\
&\cP(z,y)=\sigma_x^{-1}\|z\|^2/2+\sigma_y\|y\|^2/2+\cG(z,y),\nn \\
&\vartheta_k=\eta_z^{-1}\|z^k-z^*\|^2+\eta_y^{-1}\|y^k-y^*\|^2+2\bar\alpha^{-1}(\cP(z^k_f,y^k_f)-\cP(z^*,y^*)), \label{ea-L} \\
&a^k_x(x,y)=\nabla_x\hat h(x,y)+\sigma_x(x-\sigma_x^{-1}z^k_g)/2,\quad a^k_y(x,y)=-\nabla_y\hat h(x,y)+\sigma_y y+\sigma_x(y-y^k_g)/8,\notag
\end{align}
where $\bar \alpha=\min\left\{1,\sqrt{8\sigma_y/\sigma_x}\right\}$, $\eta_z=\sigma_x/2$, $\eta_y=\min\left\{1/(2\sigma_y),4/(\bar\alpha\sigma_x)\right\}$, and $y^k$, $y^k_f$, $y^k_g$, $z^k$, $z^k_f$ and $z^k_g$ are generated at iteration $k$ of Algorithm \ref{mmax-alg1} below. By Assumption~\ref{ea}, one can observe that $D_\bfx$, $D_\bfy$ and $\H_{\rm low}$ are finite. 

We are now ready to review a modified optimal first-order method \cite[Algorithm 1]{lu2024first} for solving \eqref{ea-prob} in Algorithm \ref{mmax-alg1}. It is a slight modification of an optimal first-order method \cite[Algorithm 4]{kovalev2022first} by incorporating a forward-backward splitting scheme and a verifiable termination criterion (see steps 23-25 in Algorithm \ref{mmax-alg1}) in order to find an $\bar\epsilon$-primal-dual stationary point of problem \eqref{ea-prob} for any prescribed tolerance $\bar\epsilon>0$. 

\begin{algorithm}[H]
\caption{A modified optimal first-order method for problem \eqref{ea-prob}}
\label{mmax-alg1}
\begin{algorithmic}[1]
\REQUIRE $\bar\epsilon>0$, $\bar z^0=z^0_f\in-\sigma_x\dom\,\np$,\footnote{} $\bar y^0=y^0_f\in\dom\,\nq$, $(z^0,y^0)=(\bar z^0, \bar y^0)$,  $\bar \alpha=\min\left\{1,\sqrt{8\sigma_y/\sigma_x}\right\}$, $\eta_z=\sigma_x/2$, $\eta_y=\min\left\{1/(2\sigma_y),4/(\bar \alpha\sigma_x)\right\}$, $\beta_t=2/(t+3)$, $\zeta=\left(2\sqrt{5}(1+8L_{\nabla\h}/\sigma_x)\right)^{-1}$, $\gamma_x=\gamma_y=8\sigma_x^{-1}$, and $\bar\zeta=\min\{\sigma_x,\sigma_y\}/L_{\nabla \h}^2$.
\FOR{$k=0,1,2,\ldots$}
\STATE $(z^k_g,y^k_g)=\bar \alpha(z^k,y^k)+(1-\bar \alpha)(z^k_f,y^k_f)$.
\STATE $(x^{k,-1},y^{k,-1})=(-\sigma_x^{-1}z^k_g,y^k_g)$.
\STATE $x^{k,0}=\prox_{\zeta\gamma_x\np}(x^{k,-1}-\zeta\gamma_x a^k_x(x^{k,-1},y^{k,-1}))$.
\STATE $y^{k,0}=\prox_{\zeta\gamma_y \nq}(y^{k,-1}-\zeta\gamma_y a^k_y(x^{k,-1},y^{k,-1}))$.
\STATE $b^{k,0}_x=\frac{1}{\zeta\gamma_x}(x^{k,-1}-\zeta\gamma_x a^k_x(x^{k,-1},y^{k,-1})-x^{k,0})$.
\STATE $b^{k,0}_y=\frac{1}{\zeta\gamma_y}(y^{k,-1}-\zeta\gamma_y a^k_y(x^{k,-1},y^{k,-1})-y^{k,0})$.
\STATE $t=0$.
\WHILE{\\ $\gamma_x\|a^k_x(x^{k,t},y^{k,t})+b^{k,t}_x\|^2+\gamma_y\|a^k_y(x^{k,t},y^{k,t})+b^{k,t}_y\|^2>\gamma_x^{-1}\|x^{k,t}-x^{k,-1}\|^2+\gamma_y^{-1}\|y^{k,t}-y^{k,-1}\|^2$\\~~}
\STATE $x^{k,t+1/2}=x^{k,t}+\beta_t(x^{k,0}-x^{k,t})-\zeta\gamma_x(a^k_x(x^{k,t},y^{k,t})+b^{k,t}_x)$.
\STATE $y^{k,t+1/2}=y^{k,t}+\beta_t(y^{k,0}-y^{k,t})-\zeta\gamma_y(a^k_y(x^{k,t},y^{k,t})+b^{k,t}_y)$.
\STATE $x^{k,t+1}=\prox_{\zeta\gamma_x \np}(x^{k,t}+\beta_t(x^{k,0}-x^{k,t})-\zeta\gamma_x a^k_x(x^{k,t+1/2},y^{k,t+1/2}))$.
\STATE $y^{k,t+1}=\prox_{\zeta\gamma_y \nq}(y^{k,t}+\beta_t(y^{k,0}-y^{k,t})-\zeta\gamma_y a^k_y(x^{k,t+1/2},y^{k,t+1/2}))$.
\STATE $b^{k,t+1}_x=\frac{1}{\zeta\gamma_x}(x^{k,t}+\beta_t(x^{k,0}-x^{k,t})-\zeta\gamma_x a^k_x(x^{k,t+1/2},y^{k,t+1/2})-x^{k,t+1})$.
\STATE $b^{k,t+1}_y=\frac{1}{\zeta\gamma_y}(y^{k,t}+\beta_t(y^{k,0}-y^{k,t})-\zeta\gamma_y a^k_y(x^{k,t+1/2},y^{k,t+1/2})-y^{k,t+1})$.
\STATE $t \leftarrow t+1$.
\ENDWHILE
\STATE $(x^{k+1}_f,y^{k+1}_f)=(x^{k,t},y^{k,t})$.
\STATE $(z^{k+1}_f,w^{k+1}_f)=(\nabla_x\hat h(x^{k+1}_f,y^{k+1}_f)+b^{k,t}_x,-\nabla_y\hat h(x^{k+1}_f,y^{k+1}_f)+b^{k,t}_y)$.
\STATE $z^{k+1}=z^k+\eta_z\sigma_x^{-1}(z^{k+1}_f-z^k)-\eta_z(x^{k+1}_f+\sigma_x^{-1}z^{k+1}_f)$.
\STATE $y^{k+1}=y^k+\eta_y\sigma_y(y^{k+1}_f-y^k)-\eta_y(w^{k+1}_f+\sigma_yy^{k+1}_f)$.
\STATE $x^{k+1}=-\sigma_x^{-1}z^{k+1}$.
\STATE $\tx^{k+1}=\prox_{\bar\zeta \np}(x^{k+1}-\bar\zeta\nabla_x\h(x^{k+1},y^{k+1}))$.
\STATE $\ty^{k+1}=\prox_{\bar\zeta \nq}(y^{k+1}+\bar\zeta\nabla_y\h(x^{k+1},y^{k+1}))$.
\STATE Terminate the algorithm and output $(\tx^{k+1},\ty^{k+1})$ if
\[
\|\bar\zeta^{-1}(x^{k+1}-\tx^{k+1},\ty^{k+1}-y^{k+1})-(\nabla \h(x^{k+1},y^{k+1})-\nabla \h(\tx^{k+1},\ty^{k+1}))\|\leq\bar\epsilon.
\]
\ENDFOR
\end{algorithmic}							
\end{algorithm}
\footnotetext{For convenience, $-\sigma_x\dom\,\np$ stands for the set $\{-\sigma_x u|u\in\dom\,\np\}$.}

The following theorem presents \emph{iteration and operation complexity} of Algorithm~\ref{mmax-alg1} for finding an $\bar\epsilon$-primal-dual stationary point of problem \eqref{ea-prob}, whose proof can be found in \cite[Section 4.1]{lu2024first}.

\begin{thm}[{\bf Complexity of Algorithm \ref{mmax-alg1}}]\label{ea-prop}
Suppose that Assumption \ref{ea} hold. Let $\H^*$, $D_\bfx$, $D_\bfy$, $\H_{\rm low}$, and $\vartheta_0$ be defined in \eqref{ea-prob}, \eqref{mmax-D1}, \eqref{ea-bnd} and \eqref{ea-L}, $\sigma_x$, $\sigma_y$ and $L_{\nabla \h}$ be given in Assumption \ref{ea}, $\bar \alpha$, $\eta_y$, $\eta_z$, $\bar\epsilon$, $\bar\zeta$ be given in Algorithm~\ref{mmax-alg1}, and 
\begin{align}
\bar \delta=&\ (2+\bar \alpha^{-1})\sigma_x D_\bfx^2+\max\{2\sigma_y,\bar \alpha\sigma_x/4\}D_\bfy^2,\nn \\ 
\bar K=&\ \left\lceil\max\left\{\frac{2}{\bar \alpha},\frac{\bar \alpha\sigma_x}{4\sigma_y}\right\}\log\frac{4\max\{\eta_z\sigma_x^{-2},\eta_y\}\vartheta_0}{(\bar\zeta^{-1}+ L_{\nabla \h})^{-2}\bar\epsilon^2}\right\rceil_+, \nn \\ 
\bar N=&\ \left\lceil\max\left\{2,\sqrt{\frac{\sigma_x}{2\sigma_y}}\right\}\log\frac{4\max\left\{1/(2\sigma_x),\min\left\{1/(2\sigma_y),4/(\bar \alpha\sigma_x)\right\}\right\}\left(\bar \delta+2\bar \alpha^{-1}\left(\H^*-\H_{\rm low}\right)\right)}{(L_{\nabla \h}^2/\min\{\sigma_x,\sigma_y\}+ L_{\nabla \h})^{-2}\bar\epsilon^2}\right\rceil_+\notag\\
&\ \times\ \left(\left\lceil96\sqrt{2}\left(1+8L_{\nabla \h}\sigma_x^{-1}\right)\right\rceil+2\right). \nn 
\end{align}
Then Algorithm~\ref{mmax-alg1} outputs an $\bar\epsilon$-primal-dual stationary point of \eqref{ea-prob} in at most $\bar K$ iterations. Moreover, the total number of evaluations of $\nabla \h$ and proximal operators of $\np$ and $\nq$ performed in Algorithm~\ref{mmax-alg1} is no more than $\bar N$, respectively.
\end{thm}


\begin{thebibliography}{10}

\bibitem{Ant21}
K.~Antonakopoulos, E.~V. Belmega, and P.~Mertikopoulos.
\newblock Adaptive extra-gradient methods for min-max optimization and games.
\newblock In {\em The International Conference on Learning Representations},
  2021.

\bibitem{BM14}
E.~G. Birgin and J.~M. Mart{\'\i}nez.
\newblock {\em Practical Augmented Lagrangian Methods for Constrained
  Optimization}.
\newblock SIAM, 2014.

\bibitem{BM20}
E.~G. Birgin and J.~M. Mart{\'\i}nez.
\newblock Complexity and performance of an augmented {L}agrangian algorithm.
\newblock {\em Optim. Methods and Softw.}, 35(5):885--920, 2020.

\bibitem{Ces06}
N.~Cesa-Bianchi and G.~Lugosi.
\newblock {\em Prediction, learning, and games}.
\newblock Cambridge University Press, 2006.

\bibitem{CGLY17}
X.~Chen, L.~Guo, Z.~Lu, and J.~J. Ye.
\newblock An augmented {L}agrangian method for non-{L}ipschitz nonconvex
  programming.
\newblock {\em SIAM J. Numer. Anal.}, 55(1):168--193, 2017.

\bibitem{chen2021proximal}
Z.~Chen, Y.~Zhou, T.~Xu, and Y.~Liang.
\newblock Proximal gradient descent-ascent: variable convergence under {K$\L$}
  geometry.
\newblock {\em arXiv preprint arXiv:2102.04653}, 2021.

\bibitem{clarke1990optimization}
F.~H. Clarke.
\newblock {\em Optimization and nonsmooth analysis}.
\newblock SIAM, 1990.

\bibitem{Dai18}
B.~Dai, A.~Shaw, L.~Li, L.~Xiao, N.~He, Z.~Liu, J.~Chen, and L.~Song.
\newblock {SBEED}: Convergent reinforcement learning with nonlinear function
  approximation.
\newblock In {\em International Conference on Machine Learning}, pages
  1125--1134, 2018.

\bibitem{dai2024optimality}
Y.-H. Dai, J.~Wang, and L.~Zhang.
\newblock Optimality conditions and numerical algorithms for a class of
  linearly constrained minimax optimization problems.
\newblock {\em SIAM Journal on Optimization}, 34(3):2883--2916, 2024.

\bibitem{dai2020optimality}
Y.-H. Dai and L.~Zhang.
\newblock Optimality conditions for constrained minimax optimization.
\newblock {\em arXiv preprint arXiv:2004.09730}, 2020.

\bibitem{dai2022rate}
Y.-H. Dai and L.-W. Zhang.
\newblock The rate of convergence of augmented lagrangian method for minimax
  optimization problems with equality constraints.
\newblock {\em Journal of the Operations Research Society of China} 12, pages
  265--297, 2024.

\bibitem{Du17}
S.~S. Du, J.~Chen, L.~Li, L.~Xiao, and D.~Zhou.
\newblock Stochastic variance reduction methods for policy evaluation.
\newblock In {\em International Conference on Machine Learning}, pages
  1049--1058, 2017.

\bibitem{Duc19}
J.~Duchi and H.~Namkoong.
\newblock Variance-based regularization with convex objectives.
\newblock {\em Journal of Machine Learning Research}, 20(1):2450–2504, 2019.

\bibitem{fu2019network}
X.~Fu and E.~Modiano.
\newblock Network interdiction using adversarial traffic flows.
\newblock In {\em IEEE INFOCOM 2019-IEEE Conference on Computer
  Communications}, pages 1765--1773. IEEE, 2019.

\bibitem{Gid19}
G.~Gidel, H.~Berard, G.~Vignoud, P.~Vincent, and S.~Lacoste{-}Julien.
\newblock A variational inequality perspective on generative adversarial
  networks.
\newblock In {\em International Conference on Learning Representations}, 2019.

\bibitem{goktas2021convex}
D.~Goktas and A.~Greenwald.
\newblock Convex-concave min-max stackelberg games.
\newblock {\em Advances in Neural Information Processing Systems},
  34:2991--3003, 2021.

\bibitem{Good14}
I.~Goodfellow, J.~Pouget-Abadie, M.~Mirza, B.~Xu, D.~Warde-Farley, S.~Ozair,
  A.~Courville, and Y.~Bengio.
\newblock Generative adversarial nets.
\newblock In {\em Advances in Neural Information Processing Systems}, pages
  2672--2680, 2014.

\bibitem{Good15}
I.~J. Goodfellow, J.~Shlens, and C.~Szegedy.
\newblock Explaining and harnessing adversarial examples.
\newblock In {\em International Conference on Learning Representations}, 2015.

\bibitem{GY19}
G.~N. Grapiglia and Y.~Yuan.
\newblock On the complexity of an augmented {L}agrangian method for nonconvex
  optimization.
\newblock {\em IMA J. Numer. Anal.}, 41(2):1508--1530, 2021.

\bibitem{guo2023fast}
Z.~Guo, Y.~Yan, Z.~Yuan, and T.~Yang.
\newblock Fast objective \& duality gap convergence for non-convex
  strongly-concave min-max problems with pl condition.
\newblock {\em J. Mach. Learn. Res.}, 24:148--1, 2023.

\bibitem{ho2023adversarial}
N.~Ho-Nguyen and S.~J. Wright.
\newblock Adversarial classification via distributional robustness with
  wasserstein ambiguity.
\newblock {\em Mathematical Programming}, 198(2):1411--1447, 2023.

\bibitem{huang2020accelerated}
F.~Huang, S.~Gao, J.~Pei, and H.~Huang.
\newblock Accelerated zeroth-order momentum methods from mini to minimax
  optimization.
\newblock {\em arXiv preprint arXiv:2008.08170}, 3, 2020.

\bibitem{KS17example}
C.~Kanzow and D.~Steck.
\newblock An example comparing the standard and safeguarded augmented
  {L}agrangian methods.
\newblock {\em Oper. Res. Lett.}, 45(6):598--603, 2017.

\bibitem{kaplan1998proximal}
A.~Kaplan and R.~Tichatschke.
\newblock Proximal point methods and nonconvex optimization.
\newblock {\em Journal of global Optimization}, 13(4):389--406, 1998.

\bibitem{kong2021accelerated}
W.~Kong and R.~D. Monteiro.
\newblock An accelerated inexact proximal point method for solving
  nonconvex-concave min-max problems.
\newblock {\em SIAM Journal on Optimization}, 31(4):2558--2585, 2021.

\bibitem{kovalev2022first}
D.~Kovalev and A.~Gasnikov.
\newblock The first optimal algorithm for smooth and
  strongly-convex-strongly-concave minimax optimization.
\newblock {\em Advances in Neural Information Processing Systems},
  35:14691--14703, 2022.

\bibitem{laidlaw2020perceptual}
C.~Laidlaw, S.~Singla, and S.~Feizi.
\newblock Perceptual adversarial robustness: Defense against unseen threat
  models.
\newblock {\em arXiv preprint arXiv:2006.12655}, 2020.

\bibitem{lin20b}
T.~Lin, C.~Jin, and M.~Jordan.
\newblock On gradient descent ascent for nonconvex-concave minimax problems.
\newblock In {\em International Conference on Machine Learning}, pages
  6083--6093, 2020.

\bibitem{lin2020near}
T.~Lin, C.~Jin, and M.~I. Jordan.
\newblock Near-optimal algorithms for minimax optimization.
\newblock In {\em Conference on Learning Theory}, pages 2738--2779, 2020.

\bibitem{lu2022single}
S.~Lu.
\newblock A single-loop gradient descent and perturbed ascent algorithm for
  nonconvex functional constrained optimization.
\newblock In {\em International Conference on Machine Learning}, pages
  14315--14357, 2022.

\bibitem{lu2020hybrid}
S.~Lu, I.~Tsaknakis, M.~Hong, and Y.~Chen.
\newblock Hybrid block successive approximation for one-sided non-convex
  min-max problems: algorithms and applications.
\newblock {\em IEEE Transactions on Signal Processing}, 68:3676--3691, 2020.

\bibitem{lu2024first}
Z.~Lu and S.~Mei.
\newblock {A first-order augmented Lagrangian method for constrained minimax
  optimization}.
\newblock {\em Mathematical Programming} 213, pages 1063--1104, 2025.

\bibitem{LZ12}
Z.~Lu and Y.~Zhang.
\newblock An augmented {L}agrangian approach for sparse principal component
  analysis.
\newblock {\em Math. Program.}, 135(1-2):149--193, 2012.

\bibitem{luo2020stochastic}
L.~Luo, H.~Ye, Z.~Huang, and T.~Zhang.
\newblock Stochastic recursive gradient descent ascent for stochastic
  nonconvex-strongly-concave minimax problems.
\newblock {\em Advances in Neural Information Processing Systems},
  33:20566--20577, 2020.

\bibitem{Madry18}
A.~Madry, A.~Makelov, L.~Schmidt, D.~Tsipras, and A.~Vladu.
\newblock Towards deep learning models resistant to adversarial attacks.
\newblock In {\em International Conference on Learning Representations}, 2018.

\bibitem{Mat10}
G.~Mateos, J.~A. Bazerque, and G.~B. Giannakis.
\newblock Distributed sparse linear regression.
\newblock {\em IEEE Transactions on Signal Processing}, 58:5262--5276, 2010.

\bibitem{Na19}
O.~Nachum, Y.~Chow, B.~Dai, and L.~Li.
\newblock {DualDICE}: Behavior-agnostic estimation of discounted stationary
  distribution corrections.
\newblock In {\em Advances in Neural Information Processing Systems}, pages
  2315--2325, 2019.

\bibitem{nou19}
M.~Nouiehed, M.~Sanjabi, T.~Huang, J.~D. Lee, and M.~Razaviyayn.
\newblock Solving a class of non-convex min-max games using iterative first
  order methods.
\newblock {\em Advances in Neural Information Processing Systems}, 32, 2019.

\bibitem{ostrovskii2021efficient}
D.~M. Ostrovskii, A.~Lowy, and M.~Razaviyayn.
\newblock Efficient search of first-order {Nash} equilibria in
  nonconvex-concave smooth min-max problems.
\newblock {\em SIAM Journal on Optimization}, 31(4):2508--2538, 2021.

\bibitem{qiu20}
S.~Qiu, Z.~Yang, X.~Wei, J.~Ye, and Z.~Wang.
\newblock Single-timescale stochastic nonconvex-concave optimization for smooth
  nonlinear td learning.
\newblock {\em arXiv preprint arXiv:2008.10103}, 2020.

\bibitem{Rak13}
A.~Rakhlin and K.~Sridharan.
\newblock Optimization, learning, and games with predictable sequences.
\newblock In {\em Advances in Neural Information Processing Systems}, pages
  3066--3074, 2013.

\bibitem{S19iAL}
M.~F. Sahin, A.~Eftekhari, A.~Alacaoglu, F.~Latorre, and V.~Cevher.
\newblock An inexact augmented {L}agrangian framework for nonconvex
  optimization with nonlinear constraints.
\newblock {\em Advances in Neural Information Processing Systems}, 32, 2019.

\bibitem{salmeron2004analysis}
J.~Salmeron, K.~Wood, and R.~Baldick.
\newblock Analysis of electric grid security under terrorist threat.
\newblock {\em IEEE Transactions on power systems}, 19(2):905--912, 2004.

\bibitem{sanj18}
M.~Sanjabi, J.~Ba, M.~Razaviyayn, and J.~D. Lee.
\newblock On the convergence and robustness of training gans with regularized
  optimal transport.
\newblock {\em Advances in Neural Information Processing Systems}, 31, 2018.

\bibitem{Sha15}
S.~Shafieezadeh-Abadeh, P.~M. Esfahani, and D.~Kuhn.
\newblock Distributionally robust logistic regression.
\newblock In {\em Advances in Neural Information Processing Systems}, page
  1576–1584, 2015.

\bibitem{Sha08}
J.~Shamma.
\newblock {\em Cooperative Control of Distributed Multi-Agent Systems}.
\newblock Wiley-Interscience, 2008.

\bibitem{Sin18}
A.~Sinha, H.~Namkoong, and J.~C. Duchi.
\newblock Certifying some distributional robustness with principled adversarial
  training.
\newblock In {\em International Conference on Learning Representations}, 2018.

\bibitem{smith2013modern}
J.~C. Smith, M.~Prince, and J.~Geunes.
\newblock Modern network interdiction problems and algorithms.
\newblock In {\em Handbook of combinatorial optimization}, pages 1949--1987.
  Springer New York, 2013.

\bibitem{song18}
J.~Song, H.~Ren, D.~Sadigh, and S.~Ermon.
\newblock Multi-agent generative adversarial imitation learning.
\newblock {\em Advances in neural information processing systems}, 31, 2018.

\bibitem{Syr15}
V.~Syrgkanis, A.~Agarwal, H.~Luo, and R.~E. Schapire.
\newblock Fast convergence of regularized learning in games.
\newblock In {\em Advances in Neural Information Processing Systems}, page
  2989–2997, 2015.

\bibitem{Tas06}
B.~Taskar, S.~Lacoste-Julien, and M.~Jordan.
\newblock Structured prediction via the extragradient method.
\newblock In {\em Advances in Neural Information Processing Systems}, page
  1345–1352, 2006.

\bibitem{tsaknakis2023minimax}
I.~Tsaknakis, M.~Hong, and S.~Zhang.
\newblock Minimax problems with coupled linear constraints: Computational
  complexity and duality.
\newblock {\em SIAM Journal on Optimization}, 33(4):2675--2702, 2023.

\bibitem{Wang21}
J.~Wang, T.~Zhang, S.~Liu, P.-Y. Chen, J.~Xu, M.~Fardad, and B.~Li.
\newblock Adversarial attack generation empowered by min-max optimization.
\newblock In {\em Advances in Neural Information Processing Systems}, 2021.

\bibitem{ward1987nonsmooth}
D.~Ward and J.~M. Borwein.
\newblock Nonsmooth calculus in finite dimensions.
\newblock {\em SIAM Journal on control and optimization}, 25(5):1312--1340,
  1987.

\bibitem{xian2021faster}
W.~Xian, F.~Huang, Y.~Zhang, and H.~Huang.
\newblock A faster decentralized algorithm for nonconvex minimax problems.
\newblock {\em Advances in Neural Information Processing Systems}, 34, 2021.

\bibitem{XW19}
Y.~Xie and S.~J. Wright.
\newblock Complexity of proximal augmented {L}agrangian for nonconvex
  optimization with nonlinear equality constraints.
\newblock {\em J. Sci. Comput.}, 86(3):1--30, 2021.

\bibitem{Xu09}
H.~Xu, C.~Caramanis, and S.~Mannor.
\newblock Robustness and regularization of support vector machines.
\newblock {\em Journal of Machine Learning Research}, 10:1485–1510, 2009.

\bibitem{XuNe05}
L.~Xu, J.~Neufeld, B.~Larson, and D.~Schuurmans.
\newblock Maximum margin clustering.
\newblock In {\em Advances in Neural Information Processing Systems}, page
  1537–1544, 2005.

\bibitem{xu2020gradient}
T.~Xu, Z.~Wang, Y.~Liang, and H.~V. Poor.
\newblock Gradient free minimax optimization: Variance reduction and faster
  convergence.
\newblock {\em arXiv preprint arXiv:2006.09361}, 2020.

\bibitem{xu2023unified}
Z.~Xu, H.~Zhang, Y.~Xu, and G.~Lan.
\newblock A unified single-loop alternating gradient projection algorithm for
  nonconvex--concave and convex--nonconcave minimax problems.
\newblock {\em Mathematical Programming}, 201(1):635--706, 2023.

\bibitem{Yang20}
J.~Yang, S.~Zhang, N.~Kiyavash, and N.~He.
\newblock A catalyst framework for minimax optimization.
\newblock In {\em Advances in Neural Information Processing Systems}, pages
  5667--5678, 2020.

\bibitem{zhang2022primal}
H.~Zhang, J.~Wang, Z.~Xu, and Y.-H. Dai.
\newblock Primal dual alternating proximal gradient algorithms for nonsmooth
  nonconvex minimax problems with coupled linear constraints.
\newblock {\em arXiv preprint arXiv:2212.04672}, 2022.

\bibitem{zhang2020single}
J.~Zhang, P.~Xiao, R.~Sun, and Z.~Luo.
\newblock A single-loop smoothed gradient descent-ascent algorithm for
  nonconvex-concave min-max problems.
\newblock {\em Advances in Neural Information Processing Systems},
  33:7377--7389, 2020.

\bibitem{zhao2024primal}
R.~Zhao.
\newblock A primal-dual smoothing framework for max-structured non-convex
  optimization.
\newblock {\em Mathematics of operations research}, 49(3):1535--1565, 2024.

\end{thebibliography}
\end{document}